\documentclass[pdflatex,sn-mathphys]{sn-jnl}% Math and Physical Sciences Reference Style
\jyear{2021}%

\theoremstyle{thmstyleone}%
\newtheorem{theorem}{Theorem}%  meant for continuous numbers
\newtheorem{corollary}[theorem]{Corollary}
\newtheorem{proposition}[theorem]{Proposition}%

\theoremstyle{thmstyletwo}%
\newtheorem{example}{Example}%
\newtheorem{remark}{Remark}%

\theoremstyle{thmstylethree}%
\newtheorem{definition}{Definition}%

\DeclareMathOperator*{\Limsup}{Lim\; sup}
\def\bd{\mbox{\rm bd}\,}

\def\epi{\mbox{\rm epi}\,}

\def\tgph{\mbox{\scriptsize {\rm gph}}\,}
\def\gph{\mbox{\rm gph}\,}
\def\dom{\mbox{\rm dom}\,}
\def\cone{\mbox{\rm cone}\,}

\def \N{\mathbb{N}}
\def \R{\mathbb{R}}
\def \B{\mathbb{B}}

\begin{document}

\title[On new generalized differentials wrt a set and applications]{On new generalized differentials with respect to a set and their applications}


\author*[1,2]{\fnm{Xiaolong} \sur{Qin}}\email{qxlxajh@163.com}

\author[2,3]{\fnm{Vo} \sur{Duc Thinh}}\email{vdthinh@dthu.edu.vn}

\author[2,4]{\fnm{Jen-Chih} \sur{Yao}}\email{yaojc@math.nsysu.edu.tw}

\affil[1]{\orgdiv{School of Mathematical Sciences}, \orgname{Hangzhou Normal University}, \orgaddress{%\street{Street},
\city{Hangzhou}, \postcode{311121},
\state{Zhejiang}, \country{China}}}

\affil[2]{\orgdiv{Department of Mathematics}, \orgname{Zhejiang Normal University}, \orgaddress{%\street{Street},
\city{Jinhua}, \postcode{321004},
\state{Zhejiang}, \country{China}}}

\affil[3]{\orgname{Dong Thap University}, \orgaddress{%\street{Street},
\city{Cao Lanh}, \postcode{81000},
\state{Dong Thap}, \country{Vietnam}}}

\affil[4]{\orgdiv{Research Center for Interneural Computing}, \orgname{China Medical University Hospital}, \orgaddress{%\street{Street},
\city{Taichung}, \postcode{40402},
 \country{Taiwan}}}


\abstract{The notions and certain fundamental characteristics of the proximal and limiting normal cones with respect to a set are first presented in this paper. We present the ideas of the limiting coderivative and subdifferential with respect to a set of multifunctions and singleton mappings, respectively, based on these normal cones. The necessary and sufficient conditions for the Aubin property with respect to a set of multifunctions are then described by using the limiting coderivative with respect to a set. As a result of the limiting subdifferential with respect to a set, we   offer the requisite optimality criteria for local solutions to optimization problems. In addition, we also provide  examples to demonstrate the outcomes.}




\maketitle

\section{Introduction}
\label{intro}
The study of the {\it well-posedness}, including the {\it Aubin property} (also known as the {\it local Lipschitz-like property}), {\it metric regularity}, etc., of multifunctions or/and the solution mapping to parameter systems, has been undoubtedly recognized as an essential problem in both variational theory and its applications \cite{AF91,BGO2017,Bonnans2000, ChuongKim-MOR-2016,DR14, G13, GO16,Ioffe2010,Mor93, Mor06, PR98, Roc96, Yen95}. These properties are defined in primal spaces and do not appear to be related to generalized differentials, however generalized differentials are essential to characterize the Aubin property, or/and the metric regularity of multifunctions. More detailed, the Aubin property of $F$ around $(\bar x,\bar y)\in \gph F$ is defined as follows:
a multifunction $F:\mathbb R^n\rightrightarrows \R^m$ is called to have the {\it Aubin property} ({\it local Lipschitz-like property}) around $(\bar x,\bar y)\in \gph F$ if there exist $\kappa>0,$ called an {\it Aubin constant}, and neighbourhoods $U$ of $\bar x,$ $V$ of $\bar y$ such that
$$F(x)\cap V\subset F(u)+\kappa\Vert x-u\Vert\mathbb B, \; \forall x,u\in U,
$$
where $\mathbb B$ is the closed unit ball of $\mathbb R^m.$ The Aubin property of a multifunction was completely characterized, in \cite[Theorem~5.7]{Mor93}, through the limiting coderivative, a type of generalized differentials for multifunctions presented by Mordukhovich et al. This result, known as the Mordukhovich criterion, states that a multifunction $F$ has the Aubin property around $(\bar x,\bar y)\in \gph F$ if and only if
$$
D^*F(\bar x,\bar y)F(0)=\{0\}.
$$
In this case, we have
$$
\mathfrak{a}F(\bar x,\bar y)=\Vert D^*F(\bar x,\bar y)\Vert:=\sup\left\{\Vert x^*\Vert\mid (y^*,x^*)\in \gph D^*F(\bar x,\bar y), y^*\in \mathbb B\right\}.
$$
where $D^*F(\bar x,\bar y)$ is the limiting coderivative mapping of $F$ at $(\bar x,\bar y)$, and $\mathfrak{a}F(\bar x,\bar y)$ is the exact Aubin constant of $F$ around $(\bar x,\bar y)$, which is the infimum of the set of Aubin constants of $F$ at that point. The Aubin property contains a   strong requirement that the value of $F(x)$ for $x$ close to $\bar x$ must be nonempty. This means that the Aubin property cannot be satisfied on the boundary points of the $\dom F$, so sometimes this property has some disadvantage in applications, such as for optimality conditions. For this reason, numerous authors have introduced generalized versions of the Aubin property, including the {\it Aubin property} (also known as {\it local Lipschitz-like property}) {\it with respect to a set}; see \cite{Roc98} for details.

Recently, by using a new type of coderivative, namely the projectional coderivative, the authors in \cite{MLYY23} established necessary and sufficient characteristics for the Aubin property with respect to a set (under the name Lipschitz-like with respect to a set) of multifunctions. In that paper, the authors  introduced the projectional subdifferential of singleton mappings and presented its applications to testing the locally Lipschitz continuous property with respect to a set of single-valued mappings.  Meanwhile, by computing (or estimating) the projectional coderivative for the solution mapping of parameter systems, the authors in \cite{YY22} provided the necessary and sufficient conditions for the Aubin property of these solution mappings under some constraint qualifications. Formulas for calculating (estimating) of the projection coderivative of multifunctions were also stated in some special cases in \cite{YMLY23}. We call the projectional coderivative and sudifferential to be projectionally generalized differentials. Although they are very useful, there are several things to keep in mind around these generalized differentials.
The first, because the projectional differentials (including the projectional coderivative and subdifferential) were defined through the projection of limiting normal cones onto tangent cones to a set, the calculation formulas, for example sum rules, of the projectional coderivative are also through the projector of sets (see \cite{YMLY23}). However, it is frequently challenging to determine the projectors of sets, particularly non-convex sets.
The second, it is known that the limiting generalized differentials are not only applied to investigate the Aubin property but also to study optimality conditions for optimization problems. A question arises  here: ``how can the projectional differentials be applied to study optimality conditions for optimization problems?''

\medskip

The motivation of our research stems from the observations above. However, we do not intend to solve these issues through the projectional differentials, but rather we will introduce a new type of generlaized differentials, generlaized differentials with respect to a set. By using these new generlaized differentials, we give equivalent characteristics of the Aubin property with respect to a set for multifunctions. In addition, we also use limiting generalized differentials with respect to a set to provide necessary optimality conditions for optimization problems with geometric constraints. One needs to know that necessary optimality conditions due to the limiting subdifferential is not as sharp as the one of the directional limiting subdifferential, a concept introduced by Ginchev and Mordukhovich in \cite{GM11}, and then developed by these authors in \cite{GM11} and by Gfrerer in \cite{G13} (see \cite{TC18,TC19} also). So, the providing necessary optimality conditions due to these new limiting generalized differentials may also obtain the similar results as the directional limiting subdifferential.

\medskip

In this work, we first present the notions and some basic properties of the proximal and limiting normal cones with respect to a set. These new normal cones with respect to a set are different from normal cones with respect to a set introduced in \cite{GM13,TC18}. On the basis of the limiting normal cone with respect to a set, we propose the concept of the limiting coderivative with respect to a set of multifunctions. We then apply the limiting coderivative with respect to a set to characterize the necessary and sufficient condition for the Aubin property with respect to a set of multifunctions. Besides, we also introduce the notions of the proximal and limiting subdifferentials with respect to a set, and we provide necessary optimality conditions for local solutions of optimization problems due to these subdifferentials. In final, we present some examples to illustrate obtained results.
\section{Preliminaries}
Through out of this paper, we assume that all of considered spaces are finite dimension spaces with Euclidean norm $\Vert\cdot\Vert$ and  scalar product $\langle \cdot,\cdot\rangle.$ Let $\mathbb R^n, \mathbb R^m, \R^s$ be Euclidean spaces with norms $\Vert \cdot\Vert_{n}$, $\Vert\cdot\Vert_m$ and $\Vert\cdot\Vert_s$, respectively. Their Cartesian product $\mathbb R^n\times \mathbb R^m$ is also equipped an Euclidean norm defined by $\Vert (x,y)\Vert:=\sqrt{\Vert x\Vert_n^2+\Vert y\Vert_m^2}.$ We sometimes use $\Vert\cdot\Vert$ for any Euclidean space if there is no the confusion. The closed unit ball in $\R^s$ is signified by $\B$, while $\B(x,r)$ is used to denote the closed ball centered at $x$ with radius $r>0.$
We define $$\R^s_+:=\left\{x=(x_1,\ldots, x_s)\in\R^s\mid x_i\ge 0\; \forall i=1,\ldots, s\right\}.$$ Given a real numbers sequence $(t_k)$, we write $t_k\to 0^+$ if $t_k\to 0$ and $t_k>0$ for all $k.$ Let $\Omega$ be a nonempty set in the finite dimension space $\mathbb R^s$ and $\bar x\in\Omega$. We denote the {\it interior} and the {\it boundary} of $\Omega$ respectively by ${\rm int}\, \Omega$ and $\bd \Omega.$ The {\it cone} of $\Omega$ is defined by $\cone(\Omega):=\R_+\Omega:=\left\{\lambda \omega\mid \lambda\ge 0, \omega\in\Omega\right\}.$
We use the notation $x_k\xrightarrow{\Omega}\bar x$ to say that $x_k\to \bar x$ and $x_k\in \Omega$ for all $k\in \N.$ The {\it distance function} to $\Omega$, $d_{\Omega}:\mathbb R^s\to \mathbb R,$ is defined by
$$
d_{\Omega}(x):=\inf_{u\in \Omega}\Vert u-x\Vert\; \text{ \rm for all } x\in \mathbb R^s.
$$
An element $u\in \Omega$ satisfying $d_{\Omega}(x)=\Vert u-x\Vert$ is called a (Euclidean) {\it projector} (or {\it closest point}) of $x$ onto $\Omega.$ The multifunction $\Pi(\cdot, \Omega): \mathbb R^s\rightrightarrows\mathbb R^s, x\mapsto \Pi(x,\Omega):=\left\{u\in \mathbb R^s\mid \Vert u-x \Vert=d_{\Omega}(x)\right\}$ is called the {\it set-valued mapping projection} onto $\Omega$, and the set $\Pi(x,\Omega)$ is called the {\it Euclidean projector set} of $x$ onto $\Omega.$ Note that $\Pi(x,\Omega)$ can be empty. However, if $\Omega$ is closed, then $\Pi(x,\Omega)\ne \emptyset$ for any $x\in \mathbb R^s.$ Given $u\in \Omega,$ we define $\Pi^{-1}(u,\Omega):=\left\{x\in \mathbb R^s\mid u\in \Pi(x,\Omega) \right\}$, which is always nonempty (in fact, $u\in \Pi^{-1}(u,\Omega)$ for all $u\in \Omega).$

Let $\bar x\in \Omega\subset \mathbb R^n$. The {\it proximal} and {\it Fréchet/regular normal cones} to $\Omega$ at $\bar x$ are respectively  given (see \cite[Definition~1.1 and page 240]{Mor06}, or \cite[page 22]{CLSW98}) by
$$
N^p(\bar x,\Omega):=\cone[\Pi^{-1}(\bar x,\Omega)-\bar x]
$$
and
$$
\hat N(\bar x,\Omega):=\left\{x^*\in \mathbb R^n\mid \limsup_{x\xrightarrow{\Omega}\bar x}\dfrac{\langle x^*, x-\bar x\rangle}{\Vert x-\bar x\Vert}\le 0\right\}.
$$
Let $\Omega$ be locally closed around $\bar x$. The {\it limiting normal cone} to $\Omega$ at $\bar x$ is defined by
$$
N(\bar x,\Omega):=\Limsup_{x\to\bar x}\Big(\cone[x-\Pi(x,\Omega)]\Big).
$$
The {\it tangent/contingent cone} to $\Omega$ at $\bar x$ is given by
$$
T(\bar x,\Omega):=\left\{v\in\R^n\mid \exists t_k\to 0^+, v_k\rightarrow\bar v \text{ \rm such that } \bar x+t_kv_k\in \Omega, \forall k\in\N \right\}.
$$

\begin{remark}\label{rem1}
{\rm (i)} We obtain  by the definition (see \cite[Example~6.16]{Roc98}) that $x^*\in N^{p}(\bar x,\Omega)$ if and only if there exists $t_0>0$ such that $\bar x\in \Pi(\bar x+t_0x^*,\Omega)$ and $\{\bar x\} = \Pi(\bar x + tx^*,\Omega)$ for all $t\in (0,t_0).$ Moreover, it is also known (see \cite[page 240]{Mor06} or \cite[Proposition~1.5]{CLSW98}) that \begin{align*}
N^p(\bar x,\Omega)&=\left\{x^*\mid \exists \delta>0 \text{ \rm such that }\langle x^*, x-\bar x\rangle\le \delta\Vert x-\bar x\Vert^2\; \forall x\in \Omega\right\}\\
&=\left\{x^*\mid \exists \delta,\delta>0 \text{ \rm such that }\langle x^*, x-\bar x\rangle\le \delta\Vert x-\bar x\Vert^2\; \forall x\in \Omega\cap\B(\bar x,\theta)\right\}.
\end{align*}

{\rm (ii)} It is known in \cite[page 240]{Mor06} that $N^p(\bar x,\Omega)\subset \hat N(\bar x,\Omega)$ and $$N(\bar x,\Omega)=\Limsup_{x\to \bar x}\hat N(x,\Omega)=\Limsup\limits_{x\to \bar x} N^p(x,\Omega).$$

  {\rm (iii)} It is also  known that $N^{p}(\bar x ,\Omega)$ is a convex cone, but it is not closed (see \cite[page~240]{Mor06}) in general.
\end{remark}

Let $f:\R^n\to\bar\R:=\R\cup\{\infty\}$ be an extended real-valued mapping. We respectively denote the domain and the epi-graph of $f$ by
$$
\dom f:=\left\{x\in\R^n\mid f(x)<\infty\right\} \text{ \rm and } \epi f:=\left\{(x,\alpha)\mid \alpha\ge f(x)\right\}.
$$
Given $\mathcal C\subset \R^n,$ we define $f_{\mathcal C}:\R^n\to\bar\R$ by $$f_{\mathcal C}(x):=\begin{cases}f(x)&\text{ \rm if } x\in \mathcal C,\\
\infty &\text{ \rm otherwise.}\end{cases}$$
The mapping $f$ is called {\it lower semi-continuous around} $\bar x$ if $\epi f$ is a locally closed set around $(\bar x,f(\bar x))$, that is, there exists $r_0>0$ such that, for any $r\in (0,r_0)$, $\epi f\cap \B((\bar x,f(\bar x)), r)$ is a closed set. In this case, we write $f\in \mathcal F(\bar x).$ We say that $f$ is {\it lower semi-continuous} and denote $f\in \mathcal F$ if $\epi f$ is closed. The mapping $f$ is called to be {\it level-bounded} if for any $\alpha\in \R$, the set $\{x\in\R^n\mid f(x)\le \alpha\}$ is bounded (possibly empty).

Consider the set-valued mapping $F:\mathbb R^n\rightrightarrows\mathbb R^m.$ We use
$$
\dom F:=\left\{x\in \R^n\mid F(x)\ne \emptyset\right\}\; \text{ \rm and } \gph F:=\left\{(x,y)\in\R^n\times\R^m\mid y\in F(x)\right\}
$$ to signify the {\it domain} and {\it graph} of $F.$ The {\it sequential Painlev\'{e}-Kuratowski upper/outer limit} of $F$ at $\bar x\in \dom F$ is given by
$$
\Limsup_{x\to\bar x}F(x):=\left\{y\in \R^m\mid \exists x_k\to \bar x, y_k\to y \text{ \rm with } y_k\in F(x_k) \; \text{ \rm for all } k\in \N\right\}.
$$

Given $\mathcal C\subset \R^n,$ we define the set-valued mapping $F_{\mathcal C}:\R^n\rightrightarrows\R^m$ as follows:
$$
F_{\mathcal C}(x):=\begin{cases}
F(x)& \text{ \rm if } x\in \mathcal C,\\
\emptyset & \text{ \rm otherwise.}
\end{cases}
$$
Then,
$$
\dom F_{\mathcal C}=\dom F\cap\mathcal C \; \text{ \rm and } \;  \gph F_{\mathcal C} = \gph F\cap(\mathcal C\times\R^m).
$$
The {\it epigraphical/profile mapping} of $F$, denoted by $\mathcal E^F:\R^n\rightrightarrows\R^m,$ is defined by
$$
\mathcal E^F(x):=F(x)+\R^m_+\; \text{ \rm for all }x\in \R^n.
$$
If $F:\R^n\rightrightarrows\R$ is a single-valued mapping, i.e., $F(x)=\{f(x)\}$ for all $x\in \R^n$, then $\gph\mathcal E^F=\epi f.$

\medskip

Based on the limiting normal cone to sets, the limiting corderivative of multifunctions and the limiting subdifferential of single-valued mappings were introduced in \cite{Roc98,Mor06}.

\begin{definition}  {\rm Let $F:\R^n \rightrightarrows\R^m$ and $f:\R^n\to \bar\R:=\R\cup\{\infty\}$.}

{\rm (i) \cite[Definition~3.32]{Mor06} The {\it limiting coderivative} (also known as the {\it normal coderivative}) of $F$ at $(\bar x,\bar y)\in \gph F$ is the multifunction $D^*F(\bar x,\bar y): \R^m\rightrightarrows\R^n$ which is defined by $$D^*F(\bar x,\bar y)(y^*):=\left\{x^*\in\R^n\mid (x^*,-y^*)\in N((\bar x,\bar y),\gph F)\right\} \; \text{ \rm for all } y^*\in \R^m.$$
}
{\rm (ii) \cite[Definition~1.77]{Mor06} The {\it limiting subdifferential}  of $f$ at $\bar x\in \dom f$ is given by $$\partial f(\bar x):=\left\{x^*\in\R^n\mid (x^*,-1)\in N((\bar x,f(\bar x)),\epi f)\right\}.$$
}
\end{definition}

A version with respect to a set of the limiting coderivative and subdifferential were presented in \cite{MLYY23,YMLY23,YY22} as the {\it projectional coderivative} and {\it subdifferential}, respectively. We recall these notions as follows.

\begin{definition}  {\rm Let $F:\R^n \rightrightarrows\R^m$ and $f:\R^n\to\bar{\R}.$ Let $\mathcal C$ be a nonempty closed set.}

{\rm (i) \cite[Definition~3.1]{MLYY23} The {\it projectional subdifferential} of $f$ at $\bar x$ with respect to $\mathcal C$, denoted by $\partial\vert_{\mathcal C}f(\bar x),$
is the set of all of $v\in \R^n$ such that there are $x_k\xrightarrow{f+\delta_{\mathcal C}}\bar x, v_k\in \partial (f+\delta_{\mathcal C})(x_k)$ with $\Pi(v_k,T(x_k,\mathcal C))\to v.$ Here, $x_k\xrightarrow{f+\delta_{\mathcal C}}\bar x$ means that $x_k\to\bar x$, $f(x_k)\to f(\bar x)$ and $x_k\in \mathcal C$ for all $k\in\N.$
}

{\rm (ii) \cite[Definition~2.2]{MLYY23} The {\it projectional coderivative} $\mathcal C$ of $F$ at $(\bar x,\bar y)\in \gph F_{\mathcal C}$ with respect to $\mathcal C$ is a multifunction $D^*\vert_{\mathcal C}F(\bar x,\bar y): \R^m\rightrightarrows \R^n,$ defined by
\begin{align*}
&x^*\in D^*\vert_{\mathcal C}F(\bar x,\bar y)(y^*) \text{ \rm if and only if } \\
&(x^*,-y^*)\in \Limsup_{(x,y)\xrightarrow{\tgph F_{\mathcal C}}(\bar x,\bar y)} \Pi(N((x,y),\gph F_{\mathcal C}), T(x,\mathcal C)\times \R^m),\end{align*} where $\Pi(A,B)$ is the projector set of $A$ onto $B$ which is the collection of projectors onto $B$ of all elements in $A.$
}
\end{definition}

The projectional coderivative (resp. projectional subdifferential) with respect to a closed convex set is an effective tool in providing necessary and sufficient conditions for the existence of the Aubin property (resp. local Lipschitz continuity) with respect to that set of multifunctions and parametric systems (resp. single-valued mappings); see \cite{MLYY23,YY22} for details.
The Aubin property with respect to a set (also known as locally Lipschitz-like property with respect to a set in \cite{MLYY23,YY22}) of $F$ was introduced in \cite[Definition~9.36]{Roc98} as follows.

\begin{definition}[\cite{Roc98}, Definition~9.36]
{\rm Let $\mathcal C$ be a nonempty and  closed subset of $\mathbb R^n$, and let $F:\mathbb R^n\rightrightarrows \R^m$. Then $F$ is said  to have the {\it Aubin property with respect to} $\mathcal C$ around $(\bar x,\bar y)\in \gph F$ if $\gph F$ is locally closed around $(\bar x,\bar y)$ and there exist $\kappa\ge 0$ and neighborhoods $V$ of $\bar y$ and $U$ of $\bar x$ such that
\begin{equation}\label{Aubin-pro}
    F(u)\cap V\subset F(x) +\kappa\Vert u-x\Vert \mathbb B\quad  \forall x,u\in \mathcal C\cap U.
\end{equation}
In the case that \eqref{Aubin-pro} holds for $\mathcal C=\dom F$, we say that $F$ has the {\it relative Aubin property} around $(\bar x,\bar y).$
}
\end{definition}

We now introduce the concept of local Lipschitz continuity with respect to set for single-valued functions as follows.

\begin{definition}[\cite{Roc98}, Definition~9.1~(b)] \label{lips-vec}
{\rm Let $\mathcal C$ be a nonempty and closed subset of $\mathbb R^n.$ Let $f:\mathbb R^n\to\bar\R$ and $\bar x\in \dom f.$ Then $f$ is said  to be {\it locally Lipschitz continuous with respect to} $\mathcal C$ around $\bar x$ if
\begin{equation}\label{lips-con}
{\rm lip}_{\mathcal C}f(\bar x):=\limsup_{x,u\xrightarrow[x\ne u]{\mathcal C}\bar x}\dfrac{\vert f(x)-f(u)\vert}{\Vert x-u\Vert}< \infty.
\end{equation}
Here ${\rm lip}_{\mathcal C}f(\bar x)$ is called the {\it exact Lipschitzian constant with respect to}  $\mathcal C$ of $f$ at $\bar x.$ If \eqref{lips-con} holds with $\mathcal C=\dom f$, then $f$ is said  to be {\it relatively locally Lipschitz continuous} around $\bar x.$
}
\end{definition}

\begin{remark}  The relatively local Lipschitz continuity of a scalar mapping is different to its local Lipschitz continuity. Similarly, the relatively Aubin property of a multifunction is also different to its Aubin property. Indeed, let $\Omega$ be a nonempty and closed subset of $\R^n.$ Its {\it scalar indicator} mapping and {\it set-valued indicator} mapping are respectively defined by
$$
\delta_{\Omega}(x):=\begin{cases}
0&\text{ \rm if } x\in\Omega;\\
\infty &\text{ \rm otherwise,}
\end{cases} \; \text{ \rm and } \; \Delta_{\Omega}(x):=\begin{cases}
\{0\}&\text{ \rm if } x\in \Omega;\\
\emptyset & \text{ \rm otherwise.}
\end{cases}
$$
Then $\delta_{\Omega}$ is relatively local Lipschitz countinuous around any $x\in \Omega$ but it is not local Lipschitz countinuous around boundary points of $\Omega.$ Similarly, $\Delta_{\Omega}$ satisfies the relative Aubin property around any $(x,0)\in \gph\Delta_{\Omega}$, but it does not satisfies Aubin property around boundary points $(\bar x,0)$ for any $\bar x\in\bd\Omega.$
\end{remark}

\section{Normal Cones with respect to a Set}

In this section, we first introduce a version with respect to a set of the proximal and limiting normal cones. These are the new concepts and are effective tools to present generalized differentials with respect to a set in the next section.

 \begin{definition}\label{normal-cone-wrt}
{\rm Let $\emptyset\ne \Omega\subset \mathbb R^n$ and $\mathcal C$ be a closed convex set. Let $\bar x\in \Omega\cap \mathcal C$, and let $\Omega$ be locally closed around $\bar x.$
}

 {\rm (i) The {\it proximal normal cone} to $\Omega$ at $\bar x$ {\it with respect to} $\mathcal C$ is defined by
\begin{align}
N^p_{\mathcal C}(\bar x,\Omega):=\big\{x^*\in \mathbb R^n\mid  \exists t>0 \text{ \rm such that }& \bar x+tx^*\in \Pi^{-1}(\bar x,\Omega\cap\mathcal C)\cap\mathcal C\big\}.
\end{align}
}

{\rm (ii) The {\it limiting normal cone} to $\Omega$ at $\bar x$ {\it with respect to} $\mathcal C$ is defined by
$$
N_{\mathcal C}(\bar x,\Omega):=\Limsup_{x\xrightarrow{\Omega\cap\mathcal C}\bar x}N^p_{\mathcal C}(x,\Omega).
$$
}
\end{definition}

\begin{figure}[!ht]
    \centering
    \includegraphics[scale=0.75]{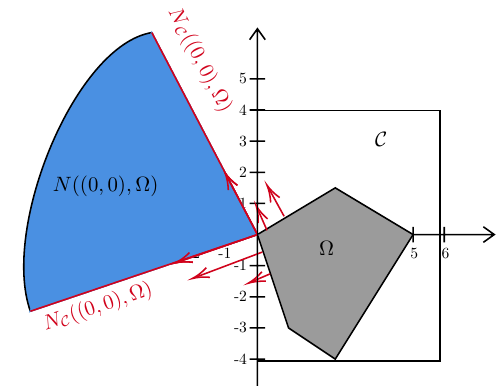}
    \caption{Illustrated example for the limiting normal cone with respect to a set}
    \label{fig:1}
\end{figure}

\begin{remark}
{\rm (i)} In fact, we do not need the local closedness of $\Omega$ in Definition~\ref{normal-cone-wrt}. However, by the applications in this work as well as to facilitate in some calculations, we need the local closedness of $\Omega.$

    {\rm (ii)} It is clear by the definition that  if $\mathcal C$ is a neighbourhood of $x$, then $N^{p}_{\mathcal C}(x, \Omega)$ coincides with $N^{p}(x, \Omega)$, so $N_{\mathcal C}(x, \Omega)=N(x,\Omega).$

    {\rm (iii)} Similar to $N^p(\bar x,\Omega),$ $N^{p}_{\mathcal C}(\bar x ,\Omega)$ is a convex cone (see Proposition~\ref{pro1}), but it is not closed in general.

    {\rm (iv)} It is very important to note that, in the generalized case, we obain
$$
\hat N_{\mathcal C}(x,\Omega):=\left\{x^*\in \mathbb R^n\mid \limsup_{x'\xrightarrow{\Omega\cap\mathcal C} x}\dfrac{\langle x^*, x'- x\rangle}{\Vert x'- x\Vert}\le 0\right\}\supset \hat N(x,\Omega)\; \forall x\in \mathcal C\cap \Omega.
$$
Moreover, we have
\begin{equation}\label{rem1-iii}
    N_{\mathcal C}(\bar x,\Omega) \ne  \Limsup_{x\xrightarrow{\Omega\cap\mathcal C}\bar x}\hat N_{\mathcal C}(x,\Omega)\; \forall \bar x\in \Omega\cap \mathcal C
\end{equation}
in general (see Example~\ref{exam0}). However, $N_{\mathcal C}(\bar x,\Omega)$ can be represented as the sequential Painlevé-Kuratowski outer/upper limit of cones $\cone[x-\Pi(x,\Omega\cap\mathcal C)]$ as $x\xrightarrow{\mathcal C}\bar x.$
\end{remark}

This result and some properties of the proximal and limiting normal cones with respect to a set are demonstrated in the following proposition.

\begin{proposition}\label{pro1}
Let $\Omega$ be a nonempty set, and let $\mathcal C$ be a nonempty and closed convex set in $\mathbb R^s.$ Suppose that $\Omega$ is locally closed around $\bar x$ for the given $\bar x\in \Omega\cap \mathcal C.$ Then the following assertions hold:

{\rm (i)} $N^p_{\mathcal C}(\bar x,\Omega)=\left\{x^*\mid \exists p>0: \bar x+px^*\in \mathcal C, \langle x^*, x-\bar x\rangle\le \dfrac{1}{2p}\Vert x-\bar x\Vert^2 \; \forall x\in \Omega\cap\mathcal C\right\}.$

{\rm (ii)} The proximal normal cone with respect to a set can be computed by
\begin{equation}\label{prox-nor-cone-eq2}
    N^p_{\mathcal C}(\bar x,\Omega)=N^{p}(\bar x,\Omega\cap\mathcal C)\cap\{x^*\mid \exists p>0 \text{ \rm such that } \bar x+px^*\in\mathcal C\}.
\end{equation}
Consequently, we have
\begin{align}\label{prox-nor-cone-eq2a}
    N^p_{\mathcal C}(\bar x,\Omega)=\Big\{x^*\mid &\exists \theta,\delta,p>0 \text{ \rm such that } \bar x+px^*\in \mathcal C\nonumber\\
    &\text{ \rm and } \langle x^*, x-\bar x\rangle\le \delta\Vert x-\bar x\Vert^2 \; \forall x\in \Omega\cap\mathcal C\cap\B(\bar x,\theta)\Big\}.
\end{align}

    {\rm (iii)} $N^p_{\mathcal C}(\bar x,\Omega)$ is a convex cone.

    {\rm (iv)} $N_{\mathcal C}(\bar x,\Omega)$ is a closed cone.

    {\rm (v)} $N_{\mathcal C}(\bar x, \Omega)= \Limsup\limits_{x\xrightarrow{\mathcal C}\bar x}\left(\cone[x-\Pi(x,\Omega\cap\mathcal C)]\right).$
    \end{proposition}

\begin{proof} {\rm (i)} Fixing $x^*\in N^p_{\mathcal C}(\bar x,\Omega),$ one sees that there exists $p>0$ such that $\bar x+px^*\in \Pi^{-1}(\bar x,\Omega\cap\mathcal C)\cap \mathcal C.$ Thus, for any $x\in \Omega\cap\mathcal C$,
$$
\Vert \bar x-\bar x-px^*\Vert^2\le \Vert x-\bar x-px^*\Vert^2,
$$
which is equivalent to
$$
p^2\Vert x^*\Vert\le \Vert x-\bar x\Vert^2- 2p\langle x^*, x-\bar x\rangle+ p^2\Vert x^*\Vert^2.
$$
This means that
$$
\langle x^*, x-\bar x\rangle\le \dfrac{1}{2p}\Vert x-\bar x\Vert^2 \; \forall x\in \mathcal C\cap\Omega.
$$

{\rm (ii)} Let $y^*\in N^p_{\mathcal C}(\bar x,\Omega).$ By the definition, we find $p>0$ such that $\bar x+py^*\in \Pi^{-1}(\bar x,\Omega\cap\mathcal C)\cap \mathcal C$, which  implies from the definition of proximal normal cones that $y^*\in N^p(\bar x,\Omega\cap\mathcal C)$. Hence, $y^*\in N^p(\bar x,\Omega\cap\mathcal C)\cap \{x^*\mid \exists p>0 \text{ \rm such that } \bar x+px^*\in\mathcal C\}$, which  means that
$$
N^p_{\mathcal C}(\bar x,\Omega)\subset N^p(\bar x,\Omega\cap\mathcal C)\cap \{x^*\mid \exists p>0 \text{ \rm such that } \bar x+px^*\in\mathcal C\}.
$$
To prove the opposite inclusion, we take $y^*\in N^p(\bar x,\Omega\cap\mathcal C)\cap \{x^*\mid \exists p>0: \bar x+px^*\in\mathcal C\}.$ There exist  by the definition $t_0>0, p_0>0$ such that $\bar x+t_0 y^*\in \Pi^{-1}(\bar x,\Omega\cap\mathcal C)$ and $\bar x+p_0y^*\in \mathcal C$, which  implies from Remark~\ref{rem1}~(i) that $\{\bar x\} =\Pi(\bar x+ty^*,\Omega\cap\mathcal C)$ for all $t\in (0,t_0).$ On the other hand, by the convexity of $\mathcal C$ and the fact that $\bar x\in \mathcal C$, we obtain that  $\lambda \bar x+(1-\lambda)(\bar x+p_0y^*)=\bar x+(1-\lambda)p_0y^*\in \mathcal C$ for any $\lambda\in (0,1),$ which  asserts that $\bar x+py^*\in \mathcal C$ for all $p\in (0,p_0).$ Picking $\bar p:=\min\{t_0,p_0\},$ we see that  $\bar x+\bar py^*\in \Pi^{-1}(\bar x,\Omega\cap\mathcal C)\cap\mathcal C$, which derives that $y^*\in N^p_{\mathcal C}(\bar x,\Omega).$ Therefore,
$$
N^p(\bar x,\Omega\cap\mathcal C)\cap \{x^*\mid \exists p>0: \bar x+px^*\in\mathcal C\}\subset N^p_{\mathcal C}(\bar x,\Omega).
$$

{\rm (iii)} We first prove that $N^{p}_{\mathcal C}(\bar x,\Omega)$ is a cone. It is clear that $0\in N^p_{\mathcal C}(\bar x,\Omega).$ Take $x^*\in N^{p}_{\mathcal C}(\bar x,\Omega)$ and $\lambda>0.$ Then there exists by (i) $p>0$ such that
$$
\bar x+px^*\in \mathcal C\; \text{ \rm and } \langle x^*, x-\bar x\rangle\le \dfrac{1}{2p}\Vert x-\bar x\Vert^2\; \forall x\in \mathcal C\cap\Omega.
$$
Thus, for any $x\in \mathcal C\cap\Omega$,
$$
\bar x+\frac{p}{\lambda}\lambda x^*\in \mathcal C\; \text{ \rm and } \langle \lambda x^*, x-\bar x\rangle = \lambda\langle x^*, x-\bar x\rangle\le \dfrac{\lambda}{2p}\Vert x-\bar x\Vert^2=\dfrac{1}{2\frac{p}{\lambda}}\Vert x-\bar x\Vert^2.
$$
Using (i) again, we conclude $\lambda x^*\in N_{\mathcal C}^{p}(\bar x,\Omega)$, so  $N_{\mathcal C}^p(\bar x,\Omega)$ is a cone.

We now prove  the convexity of $N^{p}_{\mathcal C}(\bar x,\Omega)$. Taking $x_1^*,x_2^*\in N^{p}_{\mathcal C}(\bar x,\Omega)$ and $\lambda\in (0,1),$ we need to obtain that $x^*:=\lambda x_1^*+(1-\lambda)x_2^*\in N^p_{\mathcal C}(\bar x,\Omega).$ Indeed, there exist, by (i), $p_1,p_2>0$ such that the following inequalities hold, for all $x\in \mathcal C\cap \Omega,$
\begin{equation*}
 \langle x_1^*,x-\bar x\rangle\le \dfrac{1}{2p_1}\Vert x-\bar x\Vert^2
\end{equation*}
and
\begin{equation*}
\langle x_2^*,x-\bar x\rangle \le \dfrac{1}{2p_2}\Vert x-\bar x\Vert^2.
\end{equation*}
These, respectively, imply that
\begin{equation}\label{pro1-eq1}
\langle \lambda x_1^*,x-\bar x\rangle\le \dfrac{\lambda}{2p_1}\Vert x-\bar x\Vert^2
\end{equation}
and
\begin{equation}\label{pro1-eq2}
\langle (1-\lambda) x_2^*,x-\bar x\rangle \le \dfrac{1-\lambda}{2p_2}\Vert x-\bar x\Vert^2.
\end{equation}
Combining \eqref{pro1-eq1} and \eqref{pro1-eq2}, we arrive at
\begin{equation*}
\langle x^*,x-\bar x\rangle\le \left(\dfrac{\lambda}{2p_1}+\dfrac{1-\lambda}{2p_2}\right)\Vert x-\bar x\Vert^2\; \forall x\in \Omega\cap\mathcal C,
\end{equation*}
which  is equivalent to, by setting $p:=\dfrac{1}{\frac{\lambda}{p_1}+\frac{1-\lambda}{p_2}},$ the following inequality
\begin{equation}\label{pro1-eq3}
\langle x^*,x-\bar x\rangle\le \dfrac{1}{2p} \Vert x-\bar x\Vert^2\; \forall x\in \Omega\cap\mathcal C.
\end{equation}
We next prove  that $\bar x+px^*\in \mathcal C.$ From the facts that $\bar x+p_1x_1^*$ and $\bar x+p_2x_2^*$   belong to the convex set $\mathcal C$, and $\dfrac{\lambda p_2}{\lambda p_2+(1-\lambda)p_1}+\dfrac{(1-\lambda) p_1}{\lambda p_2+(1-\lambda)p_1}=1,$ we conclude
\begin{equation*}
\dfrac{\lambda p_2}{\lambda p_2+(1-\lambda)p_1}(\bar x+p_1x_1^*)+\dfrac{(1-\lambda) p_1}{\lambda p_2+(1-\lambda)p_1}(\bar x+p_2x_2^*)\in \mathcal C,
\end{equation*}
which implies that
\begin{equation}\label{pro1-eq4}
\bar x+px^* \in \mathcal C.
\end{equation}
Combining \eqref{pro1-eq3} and \eqref{pro1-eq4}, we obtain  $x^*\in N^{p}_{\mathcal C}(\bar x,\Omega)$. Hence, $N^{p}_{\mathcal C}(\bar x,\Omega)$ is convex.

{\rm (iv)} We first prove $N_{\mathcal C}(\bar x,\Omega)$ is a cone. Take $x^*\in N_{\mathcal C}(\bar x,\Omega)$ and $\lambda>0.$ There exist by the definition sequences $x_{k}\xrightarrow{\mathcal C\cap\Omega}\bar x$ and $x_k^*\to x^*$ such that $x_k^*\in N^p_{\mathcal C}(x_k,\Omega)$ for all $k\in\N,$ which  implies that $\lambda x_k^*\in N^p_{\mathcal C}(x_k,\Omega)$ for all $k\in \N$ due to the conical property of $N^p_{\mathcal C}(x_k,\Omega)$ (see (iii)). Since $\lambda x_k^*\to \lambda x^*,$ we obtain  $\lambda x^*\in N_{\mathcal C}(\bar x,\Omega)$. Hence, $N_{\mathcal C}(\bar x,\Omega)$ is a cone. Now, we prove that $N_{\mathcal C}(\bar x,\Omega)$ is closed. To see this, we take a sequence $(x_k^*)\subset N_{\mathcal C}(\bar x,\Omega)$ satisfying $x_k^*\to x^*.$ We  demonstrate that $x^*\in N_{\mathcal C}(\bar x,\Omega)$ also. For each $k\in \N,$ there exists, by the definition, a sequence $x_{k_m}\xrightarrow{\Omega\cap\mathcal C} \bar x, x_{k_m}^*\to x_k^*$ such that $x_{k_m}^*\in N^p_{\mathcal C}(x_{k_m},\Omega).$ For each $k\in \N$, we take $\tilde x_k:=x_{k_k}$ and $\tilde x_k^*:=x_{k_k}^*.$ Thus $\tilde x_{k}\xrightarrow{\Omega\cap\mathcal C} \bar x$, $\tilde x_k^*\to x^*$, and $\tilde x_k^*\in N^p_{\mathcal C}(\tilde x_{k},\Omega).$ This gives us the assertion that $x^*\in N_{\mathcal C}(\bar x,\Omega)$, so $N_{\mathcal C}(\bar x,\Omega)$ is closed.

{\rm (v)} Let $x^*\in N_{\mathcal C}(\bar x, \Omega).$ There exist, by the definition, sequences $x_k\xrightarrow{\mathcal C\cap \Omega} \bar x$ and $x_k^*\to x^*$ such that $x_k^*\in N^p_{\mathcal C}(x_k,\Omega)$ for all $k\in \mathbb N.$  This follows that, for each $k\in \N,$ there exists $t_k>0$ such that $\tilde x_k:=x_k+t_kx_k^*\in \Pi^{-1}(x_k,\Omega\cap\mathcal C)\cap \mathcal C$, which means that
$$
x_k^* =  \frac{\tilde x_k-x_k}{t_k}\in \cone[\tilde x_k-\Pi(\tilde x_k,\Omega\cap\mathcal C)].
$$
Thus $x^*\in \Limsup\limits_{x\xrightarrow{\mathcal C}\bar x}\left[\cone(x-\Pi(x,\Omega\cap\mathcal C))\right]$, which gives us that
$$
N_{\mathcal C}(\bar x,\Omega)\subset \Limsup\limits_{x\xrightarrow{\mathcal C}\bar x}\left[\cone(x-\Pi(x,\Omega\cap\mathcal C))\right].
$$
To prove the opposite inclusion, we take $x^*\in \Limsup\limits_{x\xrightarrow{\mathcal C}\bar x}\left[\cone(x-\Pi(x,\Omega\cap\mathcal C))\right].$ Then we find $x_k\xrightarrow{\mathcal C} \bar x, x_k^*\to x^*$ such that $x_k^*\in \cone[x_k-\Pi(x_k,\Omega\cap\mathcal C)].$ If $x_k^*=0$, then $x_k^*\in N^p_{\mathcal C}(x_k,\Omega)$. If $x_k^*\ne 0$, then by $x_k\xrightarrow{\mathcal C}\bar x$ and the local closedness around $\bar x$ of $\Omega$ for sufficiently small $\epsilon>0,$ we have that  $\B(\bar x,\epsilon)\cap \Omega$ is a closed set and $x_k\in \B(\bar x, \epsilon)\cap\mathcal C$ for all $k$ large enough.  Since $x_k^*\in \cone[x_k-\Pi(x_k,\Omega\cap\mathcal C)],$ there exist $\tilde x_k\in \Pi(x_k,\Omega\cap\mathcal C)$ and $\alpha_k> 0$ such that $x^*_k=\alpha_k(x_k-\tilde x_k).$ Hence, one has $\tilde x_k+\frac{1}{\alpha_k}x_k^*=x_k\in \Pi^{-1}(\tilde x_k,\Omega\cap\mathcal C)\cap\mathcal C.$ It follows from  the definition that  $x_k^*\in N^p_{\mathcal C}(\tilde x_k,\Omega).$ On the other hand, since $x_k\in \B(\bar x, \epsilon)$ and $x_k\in \Pi^{-1}(\tilde x_k,\Omega\cap\mathcal C)$, one has $\Vert x_k-\tilde x_k\Vert\le \Vert x_k-\bar x\Vert\le \epsilon,$ which implies, by the triangular inequality, that $\Vert \tilde x_k-\bar x\Vert\le 2\Vert x_k-\bar x\Vert\le 2\epsilon$, which converges to $0$ as $\epsilon\to 0.$ Thus $x^*\in N_{\mathcal C}(\bar x,\Omega),$ which  means that
$$
N_{\mathcal C}(\bar x,\Omega)\supset \Limsup\limits_{x\xrightarrow{\mathcal C}\bar x}\left[\cone(x-\Pi(x,\Omega\cap\mathcal C)\right].
$$
Hence, the proof is completed.
\end{proof}

Although Proposition~\ref{pro1}~(ii) demonstrates that $N^p_{\mathcal C}(x,\Omega)=N^{p}(x,\Omega\cap\mathcal C)\cap\{x^*\mid \exists p>0: x+px^*\in\mathcal C\}$, we do not have $N_{\mathcal C}(\bar x,\Omega)=N(x,\Omega\cap\mathcal C)\cap\{x^*\mid \exists p>0: x+px^*\in\mathcal C\}.$  We close this section with an example, which illustrates this assertion and  assertion \eqref{rem1-iii}, and we also demonstrate that $N_{\mathcal C}(\bar x,\Omega)$ maybe not convex in general.

\begin{example}\label{exam0}
Consider the set $\Omega$   given by $\Omega:=\left\{(x,y)\in \R^2\mid y=\dfrac{2}{1+e^{-x}}, x\ge 0\right\}.$ Let $\mathcal C:=\R_+\times \R$ and $(\bar x,\bar y)=(0,0).$

By the directly calculating, we have  $\Omega\cap\mathcal C=\Omega$,
$$
N_{\mathcal C}((0,0),\Omega)=\{(x,y)\mid 0\le x\le-y\}\cup [\R_+(-1,1)],
$$
and
$$
N((0,0),\Omega)=\hat N((0,0),\Omega)=\{(x,y)\mid x\le-y\}.
$$
Hence
\begin{align*}N((0,0),\Omega)\cap\{(x,y)\mid \exists p>0:p(x,y)\in \R_+\times\R\} &= \{(x,y)\mid 0\le x\le-y\}\\
&\ne N_{\mathcal C}((0,0),\Omega).
\end{align*}
In addition, it also demonstrated that
$$
N_{\mathcal C}((0,0),\Omega)\subsetneq \hat N((0,0),\Omega)\subset \Limsup\limits_{(x,y)\xrightarrow{\mathcal C}(0,0)}\hat N((x,y),\Omega).
$$
Moreover, we also see that $N_{\mathcal C}((0,0),\Omega)$ is not a convex set.

 \end{example}

\section{Limiting Coderivative with respect to a Set and its Applications}

As it was already mentioned, the projectional coderivative is an effective tool in providing necessary and sufficient conditions for the existence of the Aubin property with respect to a closed convex set of multifunctions and parametric systems. However, the projection onto a set's tangent cone makes the projectional coderivative appear reluctant. The limiting coderivative with respect to a set is a brand-new limiting coderivative type that we are now introducing. Because the notation of the limiting coderivative with respect to a set is defined in a manner comparable to that of the limiting coderivative, they share many similar properties.

\begin{definition}\label{coderivative-wrt}
{\rm Let $F:\mathbb R^n\rightrightarrows \R^m$ and $(\bar x,\bar y)\in \gph F.$ Let $\emptyset\ne \mathcal C$ be a closed subset of $\dom F$ with $\bar x\in \mathcal C.$ The {\it limiting coderivative with respect to} $\mathcal C$ of $F$ at $(\bar x,\bar y)$ is a multifunction $D_{\mathcal C}^*F(\bar x,\bar y):\mathbb R^m\rightrightarrows \mathbb R^n$ defined by
$$
D^*_{\mathcal C}F(\bar x,\bar y)(y^*) = \left\{ x^*\in\mathbb R^n\mid (x^*,-y^*)\in N_{\mathcal C\times \mathbb R^m}((\bar x,\bar y),\gph F_{\mathcal C})\right\}\; \forall y^*\in \mathbb R^m.
$$
In the case of $\mathcal C=\dom F,$ we use $\bar D^*F(\bar x,\bar y)$   instead of $D^*_{\mathcal C}F(\bar x,\bar y)$ and call the {\it relative limiting coderivative} of $F$ at $(\bar x,\bar y).$
}
\end{definition}

The following two properties are directly  from the definition.

\begin{proposition}\label{pro3}
Give a nonempty, closed, and  convex set $\mathcal C\subset \mathbb R^n$ and $\bar x\in \mathcal C$. Let $F:\mathbb R^n\rightrightarrows \R^m$ have a locally closed graph around $(\bar x,\bar y)\in \gph F.$ Then

{\rm (i)} $D^*_{\mathcal C}F(x,y)$ is positively homogeneous, i.e., $\gph D^*_{\mathcal C}$ is a cone.

{\rm (ii)} $\bar D^*F(x,y)(y^*)\subset D^*F(x,y)(y^*)$ for any $y^*\in \mathbb R^m.$
\end{proposition}

Now, it is the right time for us to present an example to differentiate the concepts in Definition \ref{coderivative-wrt} from some that already exist as well as to illustrate Proposition~\ref{pro3}~(ii).

\begin{example}[\cite{YMLY23}, Example~2.1]\label{exam1}
Consider the set-valued mapping $F:\mathbb R\rightrightarrows\mathbb R$ given by
$$
F(x):=\begin{cases}
(-\infty, \dfrac{2}{1+e^{-x}}-1] & \text{ \rm if } x\ge 0,\\
\emptyset& \text{ \rm otherwise},\end{cases}
$$
and $\mathcal C:=\dom F=\mathbb R_+.$ Take $(\bar x,\bar y):=(0,0).$ It is known in \cite[Example~2.1]{YMLY23} that
$$
\gph D^*F(\bar x,\bar y) =\left\{(y^*,x^*)\in \mathbb R^2\mid 0\le y^*\le -x^*\right\}
$$
and
$$
\gph D^*\vert_{\mathcal C}F(\bar x,\bar y) =(\R_+\times\{0\})\cup \R_+(1,-1).
$$
We now compute $\bar D^*F(\bar x,\bar y)$ according to Proposition~\ref{pro1}~(iv). Observe
$$
N_{\mathcal C\times \R}((\bar x,\bar y),\gph F_{\mathcal C})=\R_+(-1,1).
$$
Thus \begin{align*}\gph \bar D^*_F(\bar x,\bar y)&=\left\{(y^*,x^*)\mid (x^*,-y^*)\in N_{\mathcal C\times \R}((\bar x,\bar y),\gph F_{\mathcal C})\right\}\\
&= \R_+(-1,-1)\subset \gph D^*F(\bar x,\bar y).
\end{align*}
\end{example}

By using the limiting coderivative with respect to a set, we now provide the necessary and sufficient conditions for the Aubin property with respect to a set of multifunctions as follows.

\begin{theorem}[A version with respect to a set of the Mordukhovich criterion] \label{thm1}
Let $\mathcal C$ be a closed and convex subset of $\mathbb R^n$. Let $F:\mathbb R^n\rightrightarrows\mathbb R^m$, $\bar x\in \mathcal C$, and $\bar y\in F(\bar x).$ Assume that $F$ has a locally closed graph around $(\bar x,\bar y)$. Then $F$ has the Aubin property with respect to $\mathcal C$ around $(\bar x,\bar y)$ if and only if the following inequality holds for any $\kappa\ge \mathfrak{a}_{\mathcal C}F(\bar x,\bar y).$
\begin{equation}\label{thm1-eqa}
\Vert x^*\Vert\le \kappa \Vert y^*\Vert\; \forall x^*\in D_{\mathcal C}^*F(\bar x,\bar y)(y^*).
\end{equation}
Consequently, $F$ has the Aubin property with respect to $\mathcal C$ around $(\bar x,\bar y)$ if and only if $D^*_{\mathcal C}F(\bar x,\bar y)(0)=\{0\}.$
Moreover, the exact Lipschitzian constant can be computed by
\begin{equation}\label{lipschit-con}
\mathfrak{a}_{\mathcal C}F(\bar x,\bar y)=\Vert D^*_{\mathcal C}F(\bar x,\bar y)\Vert:=\sup\left\{\Vert x^*\Vert\mid y^*\in \B, x^*\in D^*F_{\mathcal C}(\bar x,\bar y)(y^*)\right\}.
\end{equation}
\end{theorem}

\begin{proof}
We first prove that  condition \eqref{thm1-eqa} is the necessary and sufficient condition for the existence of the Aubin property around $(\bar x,\bar y)$ with respect to $\mathcal C$ of $F.$ To demonstrate the necessary condition, we assume that $F$ has the Aubin property with respect to $\mathcal C$ around $(\bar x,\bar y)$ and then we   prove that, for any $\kappa\ge \mathfrak{a}_{\mathcal C}F(\bar x,\bar y),$
\begin{equation}\label{thm-eqa1}
    \Vert x^*\Vert \le \kappa. \Vert y^*\Vert \; \forall (y^*,x^*)\in \gph D^*_{\mathcal C}F(\bar x,\bar y).
\end{equation}
Pick $x^*\in D^*_{\mathcal C}F(\bar x,\bar y)(y^*)$ with some $y^*\in \mathbb R^m$. Using the definition, one sees that there exist sequences $(x_k,y_k)\xrightarrow{\tgph F_{\mathcal C}} (\bar x,\bar y)$ and $(x_k^*,y_k^*)\to (x^*,y^*)$ such that
$$
(x_k^*,-y_k^*)\in N^p_{\mathcal C\times \mathbb R^m}((x_k,y_k),\gph F).
$$
It follows from Proposition~\ref{pro1} that there exists $\alpha_k>0$ for every $k\in \N$ satisfying
\begin{equation}\label{thm1-eq3a}
(x_k,y_k)+\alpha_k(x_k^*,-y_k^*)\in \mathcal C\times \R^m
\end{equation}
and the following inequality holds for all $(x,y)\in \gph F_{\mathcal C}$
\begin{equation}\label{thm1-eq3}
\langle x^*, x-x_k \rangle -\langle y_k^*, y-y_k\rangle\le \dfrac{1}{2\alpha_k}\Vert (x,y)-(x_k,y_k)\Vert^2.
\end{equation}

Besides, from the Aubin property with respect to $\mathcal C$ around $(\bar x,\bar y)$ with the constant $\kappa$ of $F$, we find neighborhoods $U$ of $\bar x$ and $V$ of $\bar y$ such that
\begin{equation}\label{thm1-eq4}
F(x)\cap V\subset F(z)+\kappa\Vert z-x\Vert \mathbb B\; \forall x,z\in U\cap \mathcal C.
\end{equation}
Take $\eta\in (0,\frac{1}{2})$ satisfying $\bar x+3\eta\mathbb B\subset U, \bar y+2\eta\mathbb B\subset V$, and pick $k_0\in \mathbb N$ such that $z_k, x_k\in \bar x+\eta\mathbb B$ and $y_k\in \bar y+\eta\mathbb B$ for all $k\ge k_0.$
For any $\epsilon>0,$ take $\lambda_k:=\min\{\eta,\frac{2\epsilon\alpha_k(1+\kappa)}{1+\kappa^2}\}\in (0,1).$ Then, for any $x\in \mathbb B(x_k,\lambda_k),$ one has  $x\in x_k+\eta\mathbb B\subset \bar x+2\eta\mathbb B \subset U$. Pick $\tilde x_k:=\lambda_kz_k+(1-\lambda_k)x_k$. Then $\tilde x_k\in \lambda_k(x_k+2\eta\B)+(1-\lambda_k)x_k\subset x_k+\lambda_k\B$ because of $\eta<\dfrac{1}{2}.$ Moreover, by the convexity of $\mathcal C,$ we obtain  $\tilde x_k\in \mathcal C.$ So $\tilde x_k\in (x_k+\lambda_k\mathbb B)\cap \mathcal C.$ Taking $x=x_k$ and $z=\tilde x_k$ in \eqref{thm1-eq4}, we obtain
\begin{align*}
    y_k&\in F(\tilde x_k)+\kappa\Vert x_k-\tilde x_k\Vert \mathbb B,
\end{align*}
which  implies the existence of $\tilde y_k\in F(\tilde x_k)$ and $b_k\in \mathbb B$ satisfying
$$
y_k =\tilde y_k + \kappa\Vert x_k-\tilde x_k\Vert b_k.
$$
Taking \eqref{thm1-eq3} into account with $x=\tilde x_k, y=\tilde y_k$, we obtain
\begin{align*}
\langle x_k^*,\tilde x_k-x_k\rangle -\langle y_k^*, \kappa\Vert\tilde x_k-x_k\Vert b_k\rangle&\le \dfrac{1}{2\alpha_k}\left(\Vert \tilde x_k-x_k\Vert^2 + \kappa^2\Vert\tilde x_k-x_k\Vert^2 \Vert b_k\Vert^2\right)\\
&\le \dfrac{1}{2\alpha_k}(1+\kappa^2)\Vert \tilde x_k-x_k\Vert^2\\
&\le \epsilon(1+\kappa)\Vert \tilde x_k-x_k\Vert.
\end{align*}
Note that $\tilde x_k-x_k=\lambda_k(z_k-x_k)=\lambda_k\alpha_kx_k^*.$ It follows that
$$
\lambda_k\langle x_k^*, \alpha_kx_k^*\rangle - \lambda_k\langle y_k^*, \kappa\alpha_k\Vert x_k^*\Vert b_k\rangle\le \lambda_k\epsilon(1+\kappa)\alpha_k\Vert x_k^*\Vert,
$$
which follows that
$$
\Vert x_k^*\Vert\le \kappa\Vert y_k^*\Vert+\epsilon(1+\kappa).
$$
Passing to the limit as $\epsilon\to 0$ and $k\to \infty,$ we further conclude
$\Vert x^*\Vert\le \kappa\Vert y^*\Vert.$

We now prove  the sufficient conditions, i.e., if   relation \eqref{thm1-eqa} holds for some $\kappa\ge 0$, then $F$ has Aubin property with respect to $\mathcal C$ around $(\bar x,\bar y)$ with $0\le \mathfrak{a}_{\mathcal C}F(\bar x,\bar y)\le\kappa.$ Indeed, if this assertion is not true, then  relation \eqref{Aubin-pro} is not satisfied for some $\tilde\kappa> \kappa.$ Then there exist $x_{k}',x_k''\xrightarrow{\tgph F_{\mathcal C}}(\bar x,\bar y)$ with $x_k'\ne x_k''$ and $y_k''\xrightarrow{F(x_k'')}\bar y$ such that
$$
d\left(y_k'',F(x_k')\right)>\tilde\kappa \Vert x_k' - x_k''\Vert,
$$
which implies that $y_k''\notin F(x_k')$. Since $F$ has a locally closed graph around $(\bar x,\bar y),$ one sees that there exists a $r>0$ such that $\gph F\cap(\B(\bar x,4r)\times\B(\bar x,2r))$ is a closed set, which  implies that $\gph F\cap(\B(\bar x,4r)\times\B(\bar x,2r))$ is compact. Without loss of generality, we may assume that $x_k',x_k''\in \B(\bar x,r)$ and $y_k''\in \B(\bar y,r)$.  Then there exists $y_k''\ne y_k'\in \Pi(y_k'';F(x_k'))$ such that
\begin{equation}\label{thm1-eq6}
\Vert y-y_k''\Vert\ge \Vert y_k'-y_k''\Vert>\tilde{\kappa} \Vert x_k''-x_k'\Vert\; \forall y\in F(x_k').
\end{equation}
 Moreover, by putting $\Omega:=\gph F_{\mathcal C}\cap\B(\bar x,4r)\times\B(\bar y,2r),$ one sees that  $\Omega$ is compact due to the closedness of $\mathcal C\times\R^m$ and the compactness of $\gph F\cap(\B(\bar x,4r)\times\B(\bar x,2r))$. We now define the extended real-valued function as follows:
$$
\varphi_k(x,y):=\tilde \kappa\Vert x - x_k'\Vert+\Vert y-y_k''\Vert +\delta_{\Omega}(x,y)\; \forall (x,y)\in \mathbb R^n\times\mathbb R^m.
$$
Then $\varphi_k$ is lower semi-continuous, level-bounded, and proper. By using \cite[Theorem~1.9]{Roc98}, there exists $(\hat x_k,\hat y_k)\in \Omega$, which is a the global minimizer to $\varphi_k.$ We next prove that $(\hat x_k,\hat y_k)\xrightarrow{\mathcal C\times \R^m} (\bar x,\bar y).$ Since $(\hat x_k,\hat y_k)\in \Omega\subset \mathcal C\times\R^m,$ it is enough to prove  that $\hat x_k\to \bar x$ and $\hat y_k\to \bar y.$ Indeed, we first see that
$$
\tilde\kappa\Vert x_k''-x_k'\Vert=\varphi_k(x_k'',y_k'')\ge \varphi_k(\hat x_k,\hat y_k)\ge \tilde\kappa\Vert \hat x_k-x_k'\Vert.
$$
By the facts that $x_k',x_k''\to\bar x$ and
\begin{align}
\Vert \hat x_k-\bar x\Vert&\le \Vert \hat x_k-x_k'\Vert +\Vert x_k'-\bar x\Vert\nonumber\\
&\le \Vert x_k''-x_k'\Vert+\Vert x_k'-\bar x\Vert\nonumber\\
&\le 2\Vert x_k'-\bar x\Vert+\Vert x_k''-\bar x\Vert,\label{thm1-eq6a}
\end{align} we obtain $\hat x_k\to \bar x.$ On the other hand, since $(\hat x_k,\hat y_k)$ is the global minimizer to $\varphi_k,$ one has $$\varphi(\hat x_k,\hat y_k)=\kappa\Vert \hat x_k-x_k'\Vert+\Vert y_k''-\hat y_k\Vert\le \varphi(x_k'',y_k'')=\kappa\Vert x_k'-x_k''\Vert,$$ which follows that $\Vert y_k''-\hat y_k\Vert\to 0$ due to $\Vert \hat x_k-x_k'\Vert\to 0$ and $\Vert x_k'-x_k''\Vert\to 0.$ According to $y_k''\to\bar y,$ we derive that $\hat y_k\to \bar y.$

We next demonstrate that $\hat x_k\ne x_k'.$ Indeed, in the opposite that $\hat x_k=x_k'$, one has $\hat y_k\in F(x_k')$. Thus
$$
\varphi_k(\hat x_k,\hat y_k)=\Vert \hat y_k-y_k''\Vert\ge \Vert y_k'-y_k''\Vert>\tilde\kappa\Vert x_k'-x_k''\Vert=\varphi_k(x_k'',y_k''),
$$
which  contradicts  the fact that $(\hat x_k,\hat y_k)$ is a global minimizer of $\varphi_k.$

We now consider the following two cases.

Case 1. $\hat y_k=y_k''.$ In this case, $\hat x_k\in F^{-1}(y_k'')\cap \mathcal C.$ Since $(\hat x_k,y_k'')$ is global minimizer of $\varphi_k$ for any $(x,y)\in \Omega=\gph F_{\mathcal C}\cap\B(\bar x,4r)\times\B(\bar y,2r),$ we have
\begin{equation}\label{thm1-eq7}
\tilde \kappa\Vert x-\hat x_k\Vert+\Vert y-y_k''\Vert\ge \tilde \kappa\Vert \hat x_k- x_k'\Vert.
\end{equation}
Using the Cauchy-Schwarz inequality, we obtain
\begin{equation}\label{thm1-eq8}
\sqrt{\tilde\kappa^2+1}.\Vert (x-\hat x_k,y-y_k'')\Vert \ge \tilde \kappa\Vert x-\hat x_k\Vert+\Vert y-y_k''\Vert\ge \tilde \kappa\Vert \hat x_k- x_k'\Vert.
\end{equation}
On the other hand, from \eqref{thm1-eq6} and \eqref{thm1-eq7} with $(x,y)=(x_k'',y_k'')$, we conclude that
$$
\Vert y_k'-y_k''\Vert\ge \tilde\kappa\Vert x_k''-x_k'\Vert\ge \tilde \kappa\Vert\hat x_k-x_k'\Vert.
$$
We next prove  that $(x_k'-\hat x_k,0)\in N^p_{\mathcal C\times \mathbb R^m}((\hat x_k,y_k''),\gph F).$ For any $(x,y)\in \gph F_{\mathcal C}\cap\B((\hat x_k,y_k''),r)$, we see from  \eqref{thm1-eq6a} that
$$
\Vert x-\bar x\Vert\le \Vert x-\hat x_k\Vert+\Vert\hat x_k-\bar x\Vert\le 4r
$$
and
$$
\Vert y-\bar y\Vert\le \Vert y- y_k''\Vert+\Vert y_k''-\bar y\Vert\le 2r.
$$
Therefore, $(x,y)\in \gph F_{\mathcal C}\cap(\B(\bar x,4r)\times\B(\bar y,2r))=\Omega$ and then
\begin{align*}
\langle (x_k'-\hat x_k,0), (x,y)- (\hat x_k,y_k'')\rangle &= \langle x_k'-\hat x_k, x-\hat x_k\rangle\\
& \le \Vert x_k'-\hat x_k\Vert.\Vert x-\hat x_k\Vert\\
&\le \dfrac{\sqrt{\tilde\kappa^2+1}}{\tilde\kappa}\Vert (x,y)-(\hat x_k,y_k'')\Vert.\Vert x-\hat x_k\Vert\\
&\le \dfrac{1}{2\frac{\tilde\kappa}{2\sqrt{\tilde\kappa^2+1}}}\Vert (x,y)-(\hat x_k,y_k'')\Vert^2,
\end{align*}
where the second inequality is from \eqref{thm1-eq8} and the final inequality holds by the fact that $\Vert (x,y)-(\hat x_k,y_k'')\Vert\ge \Vert x-\hat x_k\Vert\ge 0.$ Thus
$$
(x_k'-\hat x_k,0)\in N^p((\hat x_k,y_k''),\gph F_{\mathcal C}).
$$
Next, we need to affirm that
$$
(\hat x_k,y_k'')+\dfrac{\tilde\kappa}{2\sqrt{\tilde\kappa^2+1}}(x'_k-\hat x_k,0)\in \mathcal C\times \R^m.
$$
It is sufficient to prove  $\hat x_k+\dfrac{\tilde\kappa}{2\sqrt{\tilde\kappa^2+1}}(x'_k-\hat x_k)\in \mathcal C.$ Indeed, we have
$$
\hat x_k+\dfrac{\tilde\kappa}{2\sqrt{\tilde\kappa^2+1}}(x'_k-\hat x_k)=\dfrac{\tilde\kappa}{2\sqrt{\tilde\kappa^2+1}}x_k'+(1-\dfrac{\tilde\kappa}{2\sqrt{\tilde\kappa^2+1}})\hat x_k\in \mathcal C,
$$
which is due to  the convexity of $\mathcal C$ and the fact that $x_k', \hat x_k\in \mathcal C.$
Using Proposition~\ref{pro1}~(ii)~\eqref{prox-nor-cone-eq2}, we obtain
$$
(x_k'-\hat x_k,0)\in N^p_{\mathcal C\times \mathbb R^m}((\hat x_k,y_k''),\gph F).
$$
By the conic property of $N^p_{\mathcal C\times \mathbb R^m}((\hat x_k,y_k''),\gph F)$ and the fact that $\hat x_k\ne x_k',$ we have
$$
\left(\dfrac{x_k'-\hat x_k}{\Vert x_k'-\hat x_k\Vert},0\right)\in N^p_{\mathcal C\times \mathbb R^m}((\hat x_k,y_k''),\gph F).
$$
Without loss of generality, we assume that $\left(\dfrac{x_k'-\hat x_k}{\Vert x_k'-\hat x_k\Vert},0\right)\to (x^*,0).$ Thus $(x^*,0)\in N_{\mathcal C}((\bar x,\bar y),\gph F)$ and $\Vert x^*\Vert=1>\tilde\kappa\Vert 0\Vert$, which contradicts  condition \eqref{thm1-eqa}.

Case 2. $\hat y_k\ne y_k''.$ In this case,   function $\phi(x,y):=\tilde\kappa\Vert x-x_k'\Vert+\Vert y-y_k''\Vert$   belongs to $\mathcal C^2$ around $(\hat x_k,\hat y_k).$ On the other hand, it is easy to see that $(\hat x_k,\hat y_k)$ is a local solution to the following problem:
$$
\min\; \phi(x,y) \text{ \rm such that } (x,y)\in \gph F_{\mathcal C}.
$$
By using \cite[Exercise 2.12]{CLSW98}, we have
$$
-\nabla\phi(\hat x_k,\hat y_k)=\left(-\tilde\kappa\dfrac{\hat x_k-x_k'}{\Vert \hat x_k-x_k'\Vert}, -\dfrac{\hat y_k-y_k''}{\Vert \hat y_k-y_k''\Vert}\right)\in N^p((\hat x_k,\hat y_k),\gph F_{\mathcal C}).
$$
By Remark~\ref{rem1}~(i), there exists $t_k>0$ such that
$$
(\hat x_k,\hat y_k)\in \Pi\left((\hat x_k,\hat y_k)+t_k\Big(\kappa\dfrac{x_k'-\hat x_k}{\Vert x_k'-\hat x_k\Vert}, \dfrac{y_k''-\hat y_k}{\Vert y_k'' -\hat y_k\Vert}\Big),\gph F_{\mathcal C}\right)
$$
and, for any $t\in (0,t_k)$,
$$
\{(\hat x_k,\hat y_k)\}= \Pi\left((\hat x_k,\hat y_k)+t\Big(\kappa\dfrac{x_k'-\hat x_k}{\Vert x_k'-\hat x_k\Vert}, \dfrac{y_k''-\hat y_k}{\Vert y_k'' -\hat y_k\Vert}\Big),\gph F_{\mathcal C}\right).
$$
Taking $0<\bar t_k<\min\left\{t_k,\dfrac{\Vert x_k'-\hat x_k\Vert}{\tilde\kappa}\right\},$  we see that
\begin{equation}\label{thm1-eq9}
(\hat x_k,\hat y_k)\in \Pi\left((\hat x_k,\hat y_k)+\bar t_k\Big(\kappa\dfrac{x_k'-\hat x_k}{\Vert x_k'-\hat x_k\Vert}, \dfrac{y_k''-\hat y_k}{\Vert y_k'' -\hat y_k\Vert}\Big),\gph F_{\mathcal C}\right).
\end{equation}
Moreover, because of the convexity of $\mathcal C$,  $\hat x_k, x_k'\in \mathcal C$, and  $0<\lambda_k:=\dfrac{\bar t\tilde\kappa}{\Vert x_k'-\hat x_k\Vert}<1$, we obtain
$$
\hat x_k+\bar t_k\tilde\kappa\dfrac{x_k'-\hat x_k}{\Vert x_k'-\hat x_k\Vert} = \lambda_k x_k'+(1-\lambda_k)\hat x_k\in \mathcal C,
$$
which implies that
\begin{equation}\label{thm1-eq10}
(\hat x_k,\hat y_k)+\bar t_k\Big(\kappa\dfrac{x_k'-\hat x_k}{\Vert x_k'-\hat x_k\Vert}, \dfrac{y_k''-\hat y_k}{\Vert y_k'' -\hat y_k\Vert}\Big)\in \mathcal C\times \R^m.
\end{equation}
Combining \eqref{thm1-eq9} and \eqref{thm1-eq10}, we obtain
$$
\Big(\kappa\dfrac{x_k'-\hat x_k}{\Vert x_k'-\hat x_k\Vert}, \dfrac{y_k''-\hat y_k}{\Vert y_k'' -\hat y_k\Vert}\Big)\in N^p_{\mathcal C\times \R^m}((\hat x_k,\hat y_k),\gph F).
$$
Without loss of generality, we assume that $\left(-\tilde\kappa\dfrac{\hat x_k-x_k'}{\Vert \hat x_k-x_k'\Vert}, -\dfrac{\hat y_k-y_k''}{\Vert \hat y_k-y_k''\Vert}\right)\to (x^*,y^*).$ Thus $(x^*,y^*)\in N_{\mathcal C\times \R^m}((\bar x,\bar y),\gph F),$ i.e., $(-y^*, x^*)\in \gph D^*_{\mathcal C}F(\bar x,\bar y),$ and satisfies
$$
\Vert x^*\Vert=\tilde \kappa > \dfrac{\kappa+\tilde\kappa}{2}=\dfrac{\kappa+\tilde\kappa}{2}\Vert y^*\Vert.
$$
Inequality \eqref{thm1-eqa} is therefore not valid with $\bar\kappa:=\dfrac{\kappa+\tilde\kappa}{2}>\kappa.$ This is a contradiction.

In both cases, there are always contradictions. So $F$ has the Aubin property with respect to $\mathcal C$ around $(\bar x,\bar y)$ for any constant $\tilde\kappa>\kappa.$ Thus $\mathfrak{a}_{\mathcal C}F(\bar x,\bar y)\le\kappa.$
Let us now prove  formula \eqref{lipschit-con}. To do this, we prove  that both $\Vert D^*_{\mathcal C}F(\bar x,\bar y)\Vert\le \mathfrak{a}_{\mathcal C}F(\bar x,\bar y)$ and $\Vert D^*_{\mathcal C}F(\bar x,\bar y)\Vert\ge \mathfrak{a}_{\mathcal C}F(\bar x,\bar y)$ hold. First, let us prove $\Vert D^*_{\mathcal C}F(\bar x,\bar y)\Vert\le \mathfrak{a}_{\mathcal C}F(\bar x,\bar y).$ Taking an arbitrarily $\kappa\ge \mathfrak{a}_{\mathcal C}F(\bar x,\bar y),$ one sees that $F$ has the Aubin property around $(\bar x,\bar y)$ with respect to $\mathcal C$, so  inequality \eqref{thm1-eqa} holds with $\kappa,$ which  means that
$$
\Vert x^*\Vert\le \kappa \Vert y^*\Vert\; \forall (y^*,x^*)\in \gph D^*_{\mathcal C}F(\bar x,\bar y).
$$
Thus, for any $y^*\in \B,$ t $\Vert x^*\Vert\le \kappa$ for any $x^*\in D^*_{\mathcal C}F(\bar x,\bar y)(y^*)$, which follows that
$$
\Vert D^*_{\mathcal C}F(\bar x,\bar y)\Vert:=\sup\left\{\Vert x^*\Vert\mid y^*\in \B, x^*\in D^*_{\mathcal C}F(\bar x,\bar y)(y^*)\right\}\le \kappa.
$$
Since $\kappa$ is arbitrarily taken in $\big[\mathfrak{a}_{\mathcal C}F(\bar x,\bar y),\infty\big),$ we obtain $\Vert D^*_{\mathcal C}F(\bar x,\bar y)\Vert\le \mathfrak{a}_{\mathcal C}F(\bar x,\bar y).$ It remains to prove $\Vert D^*_{\mathcal C}F(\bar x,\bar y)\Vert\ge \mathfrak{a}_{\mathcal C}F(\bar x,\bar y).$ Suppose by the contradiction that $\Vert D^*_{\mathcal C}F(\bar x,\bar y)\Vert< \mathfrak{a}_{\mathcal C}F(\bar x,\bar y).$ Taking $\kappa:=\frac{\Vert D^*_{\mathcal C}F(\bar x,\bar y)\Vert+ \mathfrak{a}_{\mathcal C}F(\bar x,\bar y)}{2},$ we obtain
$$
\Vert D^*_{\mathcal C}F(\bar x,\bar y)\Vert<\kappa< \mathfrak{a}_{\mathcal C}F(\bar x,\bar y),
$$
so $F$ does not have the Aubin property around $(\bar x,\bar y)$ with respect to $\mathcal C$ with constant $\kappa.$ Therefore, there exists $(y^*,x^*)\in \gph D^*_{\mathcal C}F(\bar x,\bar y)$ such that
$$
\Vert x^*\Vert>\kappa\Vert y^*\Vert.
$$
If $y^*=0$, then by setting $\tilde x^*=(\kappa+1)\frac{x^*}{\Vert x^*\Vert},$ we obtain
$$
(0,\tilde x^*)\in \gph D^*_{\mathcal C}F(\bar x,\bar y)\;  \text{ \rm and }\; \Vert\tilde x^*\Vert=\kappa+1>\kappa.
$$
It follows that $$\Vert D^*_{\mathcal C}F(\bar x,\bar y)\Vert:=\sup\{\Vert x^*\Vert \mid y^*\in \B, (y^*,x^*)\in \gph D^*_{\mathcal C}F(\bar x,\bar y)\}\ge \Vert\tilde x^*\Vert>\kappa,$$ which is a contradiction.
If $y^*\ne 0$, then by setting $(\tilde y^*, \tilde x^*):=\frac{(y^*,x^*)}{\Vert y^*\Vert},$ we have $(\tilde y^*,\tilde x^*)\in \gph D^*_{\mathcal C}F(\bar x,\bar y)$, which satisfies
$$
\Vert \tilde x^*\Vert = \left\Vert \frac{x^*}{\Vert y^*\Vert}\right\Vert>\kappa =\kappa\Vert \tilde y^*\Vert.
$$
This implies that $\Vert D^*_{\mathcal C}F(\bar x,\bar y)\Vert:=\sup\{\Vert x^*\Vert \mid y^*\in \B, (y^*,x^*)\in \gph D^*_{\mathcal C}F(\bar x,\bar y)\}>\kappa$, which is a contradiction. Therefore, $\Vert D^*_{\mathcal C}F(\bar x,\bar y)\Vert\ge \mathfrak{a}_{\mathcal C}F(\bar x,\bar y).$
 Hence, the theorem is completely proven.  $\hfill\Box$
\end{proof}

\medskip

By taking $\mathcal C:=\dom F$ in Theorem~\ref{thm1}, we obtain a necessay and sufficient for a multifunction having the relative Aubin property as follows.

\begin{corollary}\label{cor1}
Let $F:\mathbb R^n\rightrightarrows\mathbb R^m$ and $(\bar x,\bar y)\in \gph F.$ Assume that $\dom F$ is a closed and convex set and $F$ has a locally closed graph at $(\bar x,\bar y)$. Then $F$ has the relative Aubin property around $(\bar x,\bar y)$ if and only if
\begin{equation}\label{cor1-eqa}
 \bar D^*F(\bar x,\bar y)(0)=\{0\}.
 \end{equation}
\end{corollary}

Similar to the projectional coderivative with respect to a set, the limiting coderviative with respect to a set is more useful than the directional limiting coderivative in identifying the Aubin property with respect to a set. To see this statement, let us reconsider the following example which was presented in \cite[Example~2.1]{MLYY23}.

\begin{example}\label{exam2}
Consider the multifunction $F:\mathbb R^2\rightrightarrows\R^2$ given by
$$
F(x,y)=\left\{(u,v)^{\top}\in\R^2_+\mid M(u,v)^{\top}+(x,y)^{\top}\in \R^2_+, \langle (u,v)^{\top}, M(u,v)^{\top}+(x,y)^{\top}\rangle=0\right\},
$$
where
$$
M:=\begin{bmatrix}
-1&0\\1&1
\end{bmatrix} \text{ \rm and } (x,y)^{\top}:=\begin{pmatrix}
x\\y
\end{pmatrix}.
$$
It follows that
\begin{align*}
    \dom F&=\R_+\times \R \; \text{ \rm and }\\
    \gph F&=\{(x,y,u,v)^{\top}\in \R^4\mid x,u,v\ge 0, x\ge u, y+u+v\ge 0 \\&\qquad \qquad \qquad \qquad \qquad u(x-u)=0, v(u+v+y)=0\}.
\end{align*}
Picking $\bar z:=(\bar x,\bar y)=(0,0), \bar w:=(\bar u,\bar v)=(0,0)$ and $\mathcal C:=\dom F$, we have $\gph F_{\mathcal C}=\gph F$. We now prove that $\bar D^*F(\bar z,\bar w)(0,0)=\{(0,0)\}$, which means that $z^*=0$ whenever $z^*\in \R^2$ with $(z^*,0)\in  N_{\mathcal C\times \R^2}((\bar z,\bar w),\gph F).$ Let $z^*\in \R^2$ satisfy $(z^*,0)\in N((\bar z,\bar w),\gph F).$ There exist by the definition sequences $(z_k,w_k):=(x_k,y_k,u_k,v_k)\to (0,0,0,0), (z_k^*,w_k^*)\to (z^*,0)$ such that \begin{align}(z_k^*,w_k^*)&\in
N^{p}_{\mathcal C\times \R^2}((z_k,w_k),\gph F)\nonumber\\
&=\left\{(z^*,w^*)\in N^{p}((z_k,w_k),\gph F)\mid \exists p>0: (z_k,w_k)+p(z^*,w^*)\in \dom F\times\R^2\right\}\nonumber\\
&=\left\{(z^*,w^*)\in N^{p}((z_k,w_k),\gph F)\mid \exists p>0: z_k+pz^*\in \R_+\times\R\right\}\nonumber\\
&\subset\left\{(z^*,w^*)\in N((z_k,w_k),\gph F)\mid \exists p>0: z_k+pz^*\in \R_+\times\R\right\}. \label{exam2-eq1}
\end{align}
Thank to the computation in \cite[Example~2.1]{MLYY23}, we obtain
\begin{equation}\label{exam2-eq2}
\gph F=\left\{(z,w)\in \R^2\times\R^2\left\vert
\begin{array}{cc}
w_i=0, (Mw+z)_i>0 & \text{ \rm if } i\in I_1;\\
w_i>0, (Mw+z)_i=0& \text{ \rm if } i\in I_2;\\
w_i=0, (Mw+z)_i=0& \text{ \rm if } i\in I_3
\end{array}\right. \right\}
\end{equation}
and
\begin{equation}\label{exam2-eq3}
N((z_k,w_k),\gph F)=\left\{(z^*,M^{\top}z^*+w^*)\left\vert
\begin{array}{cc}         (z^*_i,w^*_i)\in \{0\}\times\R & \text{ \rm if } i\in I_1;\\
(z_i^*,w_i^*)\in \R\times\{0\} & \text{ \rm if } i\in I_2;\\
(z^*_i,w^*_i)\in\Omega& \text{ \rm if } i\in I_3        	\end{array}\right. \right\},
\end{equation}
where $\Omega:=(\R\times\{0\})\cup (\{0\}\times\R)\cup\R^2_-$ and $I_1\cup I_2\cup I_3=\{1,2\}$ with $I_i\cap I_j=\emptyset$ for all $i,j\in\{1,2,3\}, i\ne j.$ By the computation, we obtain
\begin{equation}\label{exam2-eq4} Mw_k+z_k=\begin{bmatrix}x_k-u_k\\u_k+v_k+y_k\end{bmatrix} \text{ \rm and } M^{\top}z^*_k+w^*_k=\begin{bmatrix}w^*_{k1}-z^*_{k1}+z^*_{k2}\\ z^*_{k2}+w^*_{k2}\end{bmatrix}.
\end{equation}

We now consider the following cases and subcases.
\begin{itemize}
\item[(1)] $I_1=\{1,2\},I_2=\emptyset,$ and $I_3=\emptyset.$ In this case, we have $u_k=0, x_k>0, v_k=0, y_k>0$, and $x_k+pz_{k1}^*\in\R_+,$ which  implies from \eqref{exam2-eq3} that $z^*_{k1}=z^*_{k2}=0$. It follows that $z^*=0.$
\item[(2)] $I_1=\{1\}, I_2=\{2\},$ and  $I_3=\emptyset.$ In this case, we have $u_k=0,x_k>u_k, v_k>0,$ $u_k+y_k=0,$ and    $x_k+p_kz_{k1}^*\in\R_+$, which  implies from \eqref{exam2-eq3} that $z^*_{k1}=0, w_{k2}^*=0.$ By \eqref{exam2-eq4}, we obtain
$$
M^{\top}z^*_k+w^*_k=\begin{bmatrix}w^*_{k1}+z^*_{k2}\\ z^*_{k2}
\end{bmatrix},
$$
which follows that $z^*_{k2}\to 0$ and $w^*_{k1}\to 0$ whenever $M^{\top}z^*_k+w^*_k\to 0$. Thus $z^*=0$.
\item[(3)] $I_1=\{1\}, I_2=\emptyset,$ and $I_3=\{2\}.$ In this case, we have $u_k=0,x+k>u_k,v_k=0,y_k=0$, and $x_k+p_kz_{k1}^*\ge 0,$ which implies from \eqref{exam2-eq3} that $z_{k1}^*=0$ and $(z_{k2}^*,w_{k2}^*)\in\Omega.$ We consider the following three subcase.
\item[(3.1)] $z_{k2}^*\in\R, w_{k2}^*=0.$ We have  $M^{\top}z^*_k+w^*_k=\begin{bmatrix}w^*_{k1}+z_{k2}^*\\ z^*_{k2}\end{bmatrix}.$ Using the similar arguments as in (2), we conclude  $z^*=0.$
\item[(3.2)] $z^*_{k2}=0, w^*_{k2}\in \R.$ In this case, it is clear that $z^*=0.$
\item[(3.3)] $z_{k2}^*\le 0, w_{k2}^*\le 0.$ By \eqref{exam2-eq4}, we have
$$
M^{\top}z^*_k+w^*_k=\begin{bmatrix}w^*_{k1}+z^*_{k2}\\ z^*_{k2}+w_{k2}^*\end{bmatrix},
$$
which implies that $z_{k2}^*\to 0$ and $w_{k2}^*\to 0$ whenever $M^{\top}z^*_k+w^*_k\to 0.$ Therefore, $z^*=0.$
\item[(4)] $I_1=\{2\}, I_2=\{1\},$ and $I_3=\emptyset.$ In this case, $v_k=0, u_k+y_k>0, u_k>0, x_k=u_k$, and $x_k+p_kz_{k1}^*\ge 0$, which   implies from \eqref{exam2-eq3} that $z^*_{k2}=0, w^*_{k1}=0.$ We obtani  by \eqref{exam2-eq4} that
$$
M^{\top}z^*_k+w^*_k=\begin{bmatrix}-z^*_{k1}\\ w_{k2}^*\end{bmatrix}.
$$
Therefore, if $M^{\top}z^*_k+w^*_k\to 0$, then $w^*_{k2}\to 0$, $z_{k1}^*\to 0$. Thus $z^*=0.$
\item[(5)] $I_1=\{2\}, I_2=\emptyset,$ and $I_3=\{1\}.$ In this case, we have $v_k=0, u_k=0,y_k>0, x_k=0$, and $x_k+p_kz_{k1}^*\ge 0$, which implies that $z_{k1}^*\ge 0.$ Moreover, we obtain  from \eqref{exam2-eq3} that  $z^*_{k2}=0.$ We consider the following three subcases.
\item[(5.1)] $w_{k1}^*=0.$ We obtain by \eqref{exam2-eq4} that
$$
M^{\top}z^*_k+w^*_k=\begin{bmatrix}-z^*_{k1}\\ w_{k2}^*\end{bmatrix}.
$$
Using the similar arguments as in  (4), we reach $z^*=0.$
\item[(5.2)] $z_{k1}^*=0.$ It is clear that $z^*=0.$
\item[(5.3)] $z_{k1}^*\le 0, w_{k1}^*\le 0.$ Since $z_{k1}^*\ge 0,$ we have  $z_{k1}^*=0$. Thus $z^*=0.$
\item[(6)] $I_1=I_3=\emptyset, I_2=\{1,2\}.$ In this case, we have $u_k=v_k=0, x_k>0, y_k>0,$ $x_k+p_kz_{k1}^*\ge~0.$ From \eqref{exam2-eq3}, we obtain $w_{k1}^*=w_{k2}^*=0.$  Taking \eqref{exam2-eq4} into account, we see that
$$
M^{\top}z^*_k+w^*_k=\begin{bmatrix}z_{k2}^*-z^*_{k1}\\ z_{k2}^*\end{bmatrix}.
$$
Thus $M^{\top}z^*_k+w^*_k\to 0$ implies $z_{k2}^*\to 0, z_{k1}^*\to 0$ and thus $z^*=0.$
\item[(7)] $I_1=\emptyset, I_2=\{1\},$ and $I_3=\{2\}.$ We have $u_k=v_k=0, y_k=0, x_k>0$, and $x_k+p_kz_{k1}^*\ge 0.$ From \eqref{exam2-eq3}, we obtain $w_{k1}^*=0.$ We consider the following three subcases.
\item[(7.1)] $w_{k2}^*=0.$ Using the similar arguments as in (6), we have $z^*=0.$
\item[(7.2)] $z_{k2}^*=0.$ Using the similar arguments as in (4), we obtain  $z^*=0.$
\item[(7.3)] $z_{k2}^*\le 0, w_{k2}^*\le 0.$ From \eqref{exam2-eq4}, we conclude
$$
M^{\top}z^*_k+w^*_k=\begin{bmatrix}z_{k2}^*-z^*_{k1}\\ z_{k2}^*+w_{k2}^*\end{bmatrix}.
$$
Thus $M^{\top}z^*_k+w^*_k\to 0$ implies $z_{k2}^*\to 0, w_{k2}^*\to 0$ and $z_{k1}^*\to 0.$ Therefore, $z^*=0.$
\item[(8)] $I_1=\emptyset, I_2=\{2\}, $ and $I_1=\{1\}.$ We have  $u_k=0,x_k=0,v_k>0,v_k+y_k=0$, and $x_k+p_kz_{k1}^*\ge 0$, which implies $z_{k1}^*\ge 0.$ Using \eqref{exam2-eq3}, we have  $w_{k2}^*=0.$ We consider the following three subcases.
\item[(8.1)] $w_{k1}^*=0.$ Using the similar arguments as in (6), we have $z^*=0.$
\item[(8.2)] $z_{k1}^*=0.$ By \eqref{exam2-eq4}, we assert
$$
M^{\top}z^*_k+w^*_k=\begin{bmatrix}z_{k2}^*+w^*_{k1}\\ z_{k2}^*\end{bmatrix}.
$$
Thus $M^{\top}z^*_k+w^*_k\to 0$ implies $z_{k2}^*\to 0$, so $z^*=0.$
\item[(8.3)] $z_{k1}^*\le 0, w_{k1}^*\le 0.$ It implies $z_{k1}^*=0$ due to $z_{k1}^*\ge 0.$ Using the similar arguments as in (8.2), we obtain  $z^*=0.$
\item[(9)] $I_1=I_2=\emptyset, I_3=\{1,2\}.$ In this case, we have $u_k=v_k=x_k=y_k=0$ and $x_k+p_kz_{k1}^*\ge 0$. Thus $z_{k1}^*\ge 0.$ We consider the following nine subcases.
\item[(9.1)] $w_{k1}^*=w_{k2}^*=0.$ Using  the similar arguments as in (8.1), we have $z^*=0.$
\item[(9.2)] $w_{k1}^*=z_{k2}^*=0.$ Using  the similar arguments as in (7.2), we see $z^*=0.$
\item[(9.3)] $z_{k1}^*=w_{k2}^*=0.$ Using  the similar arguments as in (8.2), we obtain  $z^*=0.$
\item[(9.4)] $z_{k1}^*=z_{k2}^*=0.$ It is clear that $z^*=0.$
\item[(9.5)] $z_{k1}^*\le 0, w_{k1}^*\le 0,$ and $z_{k2}^*=0.$ Combining the fact that $z_{k1}^*\ge 0$, we assert  $z_{k1}^*=0$. Thus $z^*=0.$
\item[(9.6)] $z_{k1}\le 0, w_{k1}\le 0,$ and $w_{k2}=0.$ Combining the fact that $z_{k1}^*\ge 0$, we have $z_{k1}^*=0$. Using  the similar arguments as in (8.2), we find $z^*=0.$
\item[(9.7)] $z_{k1}\le 0, w_{k1}\le 0, z_{k2}^*\le 0,$ and $w_{k2}\le 0.$ Combining the fact that $z_{k1}^*\ge 0$, we have $z_{k1}^*=0$. Using  the similar arguments as in (3.3), we obtain $z^*=0.$
\item[(9.8)] $z_{k1}^*=0, z_{k2}^*\le 0, w_{k2}^*\le 0.$ Using  the similar arguments as in (3.3), we see that  $z^*=0.$
\item[(9.9)] $z_{k2}^*=0, z_{k1}^*\le 0, w_{k1}^*\le 0.$ Combining the fact that $z_{k1}^*\ge 0$, we conclude that  $z_{k1}^*=0$. Thus $z^*=0.$
\end{itemize}
All the cases and subcases above demonstrate  that $$z^*=0  \text{ \rm whenever } (z^*,0)\in \bar N_{\mathcal C\times\R^2}((\bar z,\bar w),\gph F).$$ Thus $\bar D^*F(\bar z,\bar w)(0,0)=\{(0,0)\}.$ Using Corollary~\ref{cor1}, we assert that $F$ has the relative Aubin property around $(\bar z,\bar w).$
\end{example}

\begin{remark}
From Example~\ref{exam2}, we see that the existence of $p>0$ satisfying $\bar x+px^*\in\mathcal C$ in the definition of the proximal normal cone (see Definition~\ref{normal-cone-wrt}) is effective. Indeed, if we drop this condition in the definition, then, in the cases of (5.3), (8.3), (9.5), (9.6), (9.7), and (9.9), we have $z^*\le 0$, which cannot obtain $z^*=0.$
\end{remark}

\section{Limiting subdifferential with respect to a set and its applications}
In this section, we first introduce the concept of subdifferentials with respect to a set. Then we use these subdifferentials to provide criterions for the Aupin property with respect to a set of profile mappings as well as to state optimality conditions for optimization problems with set constraints.

\begin{definition}
{\rm Consider the extended real-valued function $f:\mathbb R^n\to \bar\R.$ Let $\mathcal C$ be a closed convex subset of $\mathbb R^n$ and $\bar x\in\mathcal C.$
}

{\rm (i) A vector $x^*\in \mathbb R^n$ is called a {\it proximal subgradient with respect to} $\mathcal C$ of $f$ at $\bar x$ if
$$
(x^*,-1)\in N^p_{\mathcal C\times\R}((\bar x,f(\bar x)),\epi f).
$$
The set of all of proximal subgradients with respect to $\mathcal C$ of $f$ at $\bar x$, denoted by $\partial^p_{\mathcal C}f(\bar x),$ is called the {\it proximal subdifferential with respect to} $\mathcal C$ of $f$ at $\bar x.$ We have
$$
\partial^p_{\mathcal C}f(\bar x):=\left\{x^*\in \mathbb R^n\mid (x^*,-1)\in N^p_{\mathcal C\times \R}((\bar x,f(\bar x)),\epi f)\right\}.
$$
}

{\rm (ii) The set
$\partial_{\mathcal C}f(\bar x):=\left\{x^*\in \mathbb R^n\mid (x^*,-1)\in N_{\mathcal C\times \R}((\bar x,f(\bar x)),\epi f)\right\}$ is called the {\it limiting subdifferential with respect to} $\mathcal C$ of $f$ at $\bar x.$ We also call every vector $x^*\in \partial_{\mathcal C}f(\bar x)$ to be a {\it limiting subgradient with respect to} $\mathcal C$ of $f$ at $\bar x.$
}

{\rm (iii) The {\it singular/horizon subdifferential with respect to} $\mathcal C$ of $f$ at $\bar x,$ denoted $\partial^{\infty}_{\mathcal C}f(\bar x)$, is defined by
$$
\partial^{\infty}_{\mathcal C}f(\bar x):=\left\{x^*\in \mathbb R^n\mid (x^*,0)\in N_{\mathcal C\times\R}((\bar x,f(\bar x)),\epi f)\right\}.
$$
}

{\rm If $\bar x\notin \dom f\cap\mathcal C$, then we put $\partial^p_{\mathcal C}f(\bar x):=\partial_{\mathcal C}f(\bar x):=\partial^{\infty}_{\mathcal C}f(\bar x):=\emptyset.$ In the case of $\mathcal C=\dom f$, we write $\bar\partial^pf(\bar x),\bar\partial f(\bar x), \bar\partial^{\infty}f(\bar x)$ instead of $\partial^p_{\mathcal C}f(\bar x), \partial_{\mathcal C}f(\bar x),\partial^{\infty}_{\mathcal C}f(\bar x)$ and call the {\it relative proximal, limiting, horizon subdifferentials}, respectively.
}
\end{definition}

\begin{proposition} \label{pro4} Let  $f:\mathbb R^n\to \bar\R$ and $\mathcal C$ be a closed and convex subset of $\mathbb R^n$. Let $\bar x\in \mathcal C.$ Then the following assertions hold:

{\rm (i)} $\partial_{\mathcal C}f(\bar x) = D^*_{\mathcal C}\mathcal E^{f}(\bar x)(1)$ and $\partial^{\infty}_{\mathcal C}f(\bar x) = D^*_{\mathcal C}\mathcal E^{f}(\bar x)(0)$.

{\rm (ii)} $x^*\in \partial^p_{\mathcal C}f(\bar x)$ if and only if there exists $p>0$ such that $\bar x+px^*\in \mathcal C$ and the following inequality holds for all $x\in \mathcal C$
\begin{equation}\label{subdiff-eq}
\langle x^*, x-\bar x\rangle -f(x)+f(\bar x)\le \dfrac{1}{2p}(\Vert x-\bar x\Vert^2+(f(x)-f(\bar x))^2).
\end{equation}
Consequently, if $\mathcal C=\dom f$, then \eqref{subdiff-eq} hold for any $x\in \R^n$, that is,
\begin{align}
\bar\partial^p f(\bar x)&=\Big\{x^*\in \mathbb R^n\mid \exists p>0 \text{ \rm satisfying } \bar x+px^*\in \mathcal C, \nonumber\\
&\langle x^*, x-\bar x\rangle -f(x)+f(\bar x)\le \dfrac{1}{2p}\big(\Vert x-\bar x\Vert^2+(f(x)-f(\bar x))^2\big)\; \forall x\in \R^n\Big\}.\label{subdiff-eq2}
\end{align}
\end{proposition}

\begin{proof}
{\rm (i)} It directly implies from the definition.

{\rm (ii)} We first demonstrate the ``if'' part. Let $x^*\in \mathbb R^n$ satisfy \eqref{subdiff-eq}. For any $(x,r)\in \epi f_{\mathcal C}$, we have $x\in \dom f\cap \mathcal C, r\ge f_{\mathcal C}(x)=f(x)$ and
$$
\langle x^*, x-\bar x\rangle -r+f(\bar x)\le \dfrac{1}{2p}\left(\Vert x-\bar x\Vert^2+(r-f(\bar x))^2\right),
$$
where $p>0$ satisfies $\bar x+px^*\in \mathcal C.$ It follows that $(\bar x, f(\bar x))+p(x^*,-1)\in \mathcal C\times \R$ and
$$
\left\langle (x^*,-1),(x,r)-(\bar x, f(\bar x))\right\rangle\le \dfrac{1}{2p}\left(\Vert (x,r)-(\bar x,f(\bar x))\Vert\right)^2\; \forall (x,r)\in \epi f_{\mathcal C},
$$ which gives us that $(x^*,-1)\in N^p_{\mathcal C}((\bar x,f(\bar x)),\epi f)$. Thus $x^*\in \partial^p_{\mathcal C} f(\bar x).$

We next prove the ``only if'' part. Let $x^*\in \partial^{p}_{\mathcal C}f(\bar x)$, which  means that $(x^*,-1)\in N^p_{\mathcal C}((\bar x,f(\bar x)),\epi f_{\mathcal C}).$ It implies from Proposition~\ref{pro1}~(i) that there exists $p>0$ such that $(\bar x, f(\bar x))+p(x^*,-1)\in \mathcal C\times \R$ and
\begin{equation}\label{pro4-eq1}
\langle (x^*,-1), (x,r)-(\bar x, f(\bar x))\rangle\le \dfrac{1}{2p}\Vert (x,r)-(\bar x,f(\bar x))\Vert^2\; \forall (x,r)\in \epi f_{\mathcal C}.
\end{equation}
For any $x\in \mathcal C,$ if $x\notin \dom f$, then $f(x)=\infty$. Thus \eqref{subdiff-eq} trivially hold. If $x\in \dom f$, then $f(x)=f_{\mathcal C}(x).$ It implies from \eqref{pro4-eq1} with $r=f_{\mathcal C}(x)=f(x)$ that there exists $p>0$ such that $\bar x+px^*\in \mathcal C$ and
$$
\langle (x^*,-1), (x,f(x))-(\bar x, f(\bar x))\rangle\le \dfrac{1}{2p}\Vert (x,f(x))-(\bar x,f(\bar x))\Vert^2,
$$
which follows \eqref{subdiff-eq}.

Now, let $\mathcal C=\dom f$. Then, for any $x\in \R^n\setminus \mathcal C,$   inequality \eqref{subdiff-eq} is trivially hold. Thus we obtain \eqref{subdiff-eq2} immediately.
\end{proof}

 \medskip

In what follows, we present some applications of the limiting subdifferential with respect to a set. The first application that interests us is the characterization of the locally Lipschitz continuity of single-valued functions. The concept of the locally Lipschitz continuity with respect to a set of a single-valued mapping was presented in Defintion~\ref{lips-vec}.

It is known that a mapping $f\in \mathcal F(\mathcal C)$ is locally Lipschitz continuous with respect to $\mathcal C\subset \dom f$ around $\bar x\in \mathcal C$ if and only if its profile mapping $\mathcal E^{f}:\mathbb R^n\rightrightarrows \R$ defined by
$$
\mathcal E^f(x):=\begin{cases}
\left\{\alpha\in \R\mid \alpha\ge f(x)\right\}& \text{ \rm if } x\in \dom f,\\
\emptyset & \text{ \rm otherwise,}
\end{cases}
$$
has the Aubin property with respect to $\mathcal C$ around $(\bar x,\bar y)$ with $\bar y=f(\bar x).$ Moreover, the following relation holds: $$\mathfrak{a}_{\mathcal C}f(\bar x)=\mathfrak{a}_{\mathcal C}\mathcal E^f(\bar x,\bar y).$$

The above mentions say that we may study the locally Lipschitz continuity with respect to $\mathcal C$ of $f$ around $\bar x$ due to the Aubin property with respect to $\mathcal C$ of $\mathcal E^f$ around $(\bar x,\bar y)$. To do this, we need to find the relation between the $\partial_{\mathcal C}f(\bar x)$ and $D^*_{\mathcal C}\mathcal E^f(\bar x,\bar y).$ The following theorem gives us that relation.

\begin{theorem}\label{thm2}
Let $f:\mathbb R^n\to \bar \R$ and let $\mathcal C$ be closed convex. Let $\bar x\in \mathcal C$ and $f\in \mathcal F(\bar x).$ Then
\begin{equation}\label{thm2-eq}
D^*_{\mathcal C}\mathcal E^f(\bar x,f(\bar x))(\lambda)=\begin{cases}
\lambda\partial_{\mathcal C}f(\bar x) & \text{ \rm if } \lambda>0;\\
\partial^{\infty}_{\mathcal C}f(\bar x) & \text{ \rm if } \lambda=0;\\
\emptyset & \text{ \rm otherwise.}
\end{cases}
\end{equation}
\end{theorem}

\begin{proof}
We consider the following three cases:

Case 1. $\lambda>0.$ In this case,
\begin{align*}
D^*_{\mathcal C}\mathcal E^f(\bar x,f(\bar x))(\lambda)&=\left\{x^*\in \R^n\mid (x^*,-\lambda)\in N_{\mathcal C\times\R}\Big((\bar x, f(\bar x)),\gph \mathcal E^f\Big)\right\}.%\label{thm2-eq1}
\end{align*}
Note that $\gph \mathcal E^f_{\mathcal C}=\left\{(x,r)\mid r\ge f(x), x\in \mathcal C\right\}=\epi f_{\mathcal C}.$ By Proposition~\ref{pro1}~(iii), we have
\begin{align*}
D^*_{\mathcal C}\mathcal E^f(\bar x,f(\bar x))(\lambda) &=\left\{x^*\in \R^n\mid (\frac{x^*}{\lambda},-1)\in N_{\mathcal C\times\R}\Big((\bar x, f(\bar x)),\epi f\Big)\right\}\\
&=\left\{x^*\in \R^n\mid \frac{x^*}{\lambda}\in \partial_{\mathcal C}f(\bar x)\right\}.%\label{thm2-eq2}
\end{align*}
Therefore, $D^*_{\mathcal C}\mathcal E^f(\bar x,f(\bar x))(\lambda) =\lambda\partial_{\mathcal C}f(\bar x).$

Case 2. $\lambda=0.$ In this case,
\begin{align*}
D^*_{\mathcal C}\mathcal E^f(\bar x,f(\bar x))(0)&=\left\{x^*\in \R^n\mid (x^*,0)\in N_{\mathcal C\times\R}\Big((\bar x, f(\bar x)),\gph \mathcal E^f\Big)\right\}\\
&=\left\{x^*\in \R^n\mid (x^*,0)\in N_{\mathcal C\times\R}\Big((\bar x, f(\bar x)),\epi f\Big)\right\}\\
&=\left\{x^*\in \R^n\mid x^*\in \partial^{\infty}_{\mathcal C}f(\bar x)\right\}.
\end{align*}
So, we obtain $D^*_{\mathcal C}\mathcal E^f(\bar x,f(\bar x))(0) =\partial^{\infty}_{\mathcal C}f(\bar x).$

Case 3. $\lambda<0.$ We assume the opposite that there exists $x^*\in D^*_{\mathcal C}\mathcal E^f(\bar x,f(\bar x))(\lambda).$ We find, by the definition, sequences $(x_k,r_k)\xrightarrow{\tgph\mathcal E^f_{\mathcal C}} (\bar x,f(\bar x)), (x_k^*,\lambda_k)\to (x^*,\lambda)$ such that
$$
(x_k^*,\lambda_k)\in N^p_{\mathcal C\times\R}((x_k,r_k),\gph\mathcal E^f).
$$
Thus, for each $k\in N$, there exists, due to Proposition~\ref{pro1}~(i), $p_k>0$ such that $(x_k,r_k)+p_k(x_k^*,-\lambda_k)\in \mathcal C\times \R$ and
\begin{equation}\label{thm2-eq3}
\langle x_k^*,x-x_k\rangle-\lambda_k(r-r_k)\le \dfrac{1}{2p_k}\left(\Vert x-x_k\Vert^2+(r-r_k)^2\right).\end{equation}
Pick $x=x_k$ and $r=r_k+\dfrac{p_k}{k}.$ Then $(x,r)\in \gph\mathcal E^f_{\mathcal C}$. Taking \eqref{thm2-eq3} into account, we get $-\lambda_k\frac{p_k}{k}\le \frac{1}{2p_k}\frac{p_k^2}{k^2}$ which is equivalent to $\lambda_k\ge 0.$ This is contradiction to $\lambda_k\to \lambda<0.$

Hence,   $D^*_{\mathcal C}\mathcal E^f(\bar x,f(\bar x))(\lambda)=\emptyset.$
\end{proof}

From Theorems~\ref{thm1} and \ref{thm2}, we obtain the following theorem and corollary.

\begin{theorem}
Let $f:\mathbb R^n\to \bar \R$ and let $\mathcal C$ be closed convex. Let $\bar x\in \mathcal C$ and $f\in \mathcal F(\bar x).$ We consider the following statements:

 {\rm (i)} $f$ is locally Lipschitz continuous with respect to $\mathcal C$ around $\bar x$.

 {\rm (ii)} There exists $\kappa>0$ such that $\Vert x^*\Vert\le \kappa$ whenever $x^*\in \partial_{\mathcal C}f(\bar x).$

 {\rm (iii)} $\partial^{\infty}_{\mathcal C}f(\bar x)=\{0\}.$

 {\rm (iv)} $\partial_{\mathcal C}f(\bar x)$ is nonempty compact set and $\mathfrak{a}_{\mathcal C}f(\bar x)=\max\{\Vert x^*\Vert\mid x^*\in \partial_{\mathcal C}f(\bar x)\}.$

 Then  {\rm (i)} $\Leftrightarrow$ {\rm (ii)} $\Leftrightarrow$ {\rm (iii)} $\implies$ {\rm (iv)}.
\end{theorem}

\begin{corollary}
Let $f:\mathbb R^n\to \bar \R$ be a lower semi-continuous. Let $\dom f$ be convex. We consider the following statements:

 {\rm (i)} $f$ is relatvie locally Lipschitz continuous around $\bar x$.

 {\rm (ii)} There exists $\kappa>0$ such that $\Vert x^*\Vert\le \kappa$ whenever $x^*\in \bar\partial f(\bar x).$

 {\rm (iii)} $\bar\partial^{\infty}f(\bar x)=\{0\}.$

 {\rm (iv)} $\bar\partial f(\bar x)$ is nonempty compact set and $\mathfrak{a}_{\dom f}f(\bar x)=\max\{\Vert x^*\Vert\mid x^*\in \bar\partial f(\bar x)\}.$

 Then   {\rm (i)} $\Leftrightarrow$ {\rm (ii)} $\Leftrightarrow$ {\rm (iii)} $\implies$ {\rm (iv)}.
\end{corollary}

 We next consider how the subdifferentials with respect to a set are applied in providing optimality conditions for optimization problems with geometric constraints. To do this, we need to prove an important result as in the following proposition.

 \begin{proposition}\label{pro5}
Let $f:\mathbb R^n\to \bar \R$ and let $\mathcal C$ be closed convex. Let $\bar x\in \mathcal C\cap\dom f$ and $f\in \mathcal F(\bar x).$ Then
\begin{align}
    \partial^p_{\mathcal C}f(\bar x)&=\Big\{x^*\mid \exists p,\delta,\theta>0: \bar x+px^*\in \mathcal C,\nonumber\\
    &\qquad \qquad \langle x^*, x-\bar x\rangle -(f(x)-f(\bar x)) \le \delta\Vert x-\bar x\Vert^2\; \forall x\in \mathcal C\cap\B(\bar x,\theta)\Big\}.\label{lem3-eqii}
\end{align}
\end{proposition}

\begin{proof}
We have
\begin{align}
\partial^p_{\mathcal C}f(\bar x)&=\left\{x^*\mid (x^*,-1)\in N^p_{\mathcal C\times \R}((\bar x,f(\bar x)),\epi f)\right\}\nonumber\\
&=\left\{x^*\mid (x^*,-1)\in N^p((\bar x,f(\bar x)),\epi f_{\mathcal C})\cap\{u\mid \exists p>0: \bar x+pu\in\mathcal C\}\right\}\label{pro5-eq1}\\
&=\left\{x^*\mid x^*\in \partial^pf_{\mathcal C}(\bar x)\cap\{u\mid \exists p>0: \bar x+pu\in\mathcal C\}\right\}\nonumber\\
&=\Big\{x^*\mid \exists p,\delta,\theta>0: \bar x+px^*\in \mathcal C,\nonumber\\
    &\qquad \qquad \langle x^*, x-\bar x\rangle -(f_{\mathcal C}(x)-f_{\mathcal C}(\bar x)) \le \delta\Vert x-\bar x\Vert^2\; \forall x\in \B(\bar x,\theta)\Big\} \label{pro5-eq2}\\
&=\Big\{x^*\mid \exists p,\delta,\theta>0: \bar x+px^*\in \mathcal C,\nonumber\\
    &\qquad \qquad \langle x^*, x-\bar x\rangle -(f(x)-f(\bar x)) \le \delta\Vert x-\bar x\Vert^2\; \forall x\in \mathcal C\cap\B(\bar x,\theta)\Big\}. \label{pro5-eq3}
\end{align}
Here  \eqref{pro5-eq1} implies from Proposition~\ref{pro1}~(ii) and \eqref{pro5-eq2} holds due to \cite[Theorem~2.5]{CLSW98}, while \eqref{pro5-eq3} is obtained by the definition of $f_{\mathcal C}.$ Hence, the proof is completed.
\end{proof}

We consider the following optimization problem:
\begin{equation}\label{optimprob1}
\min f(x) \text{ \rm such that } x\in \mathcal C,
\end{equation}
where $f:\mathbb R^n\to \bar\R$ and $\mathcal C$ is a closed convex set.

An element $\bar x\in \dom f$ is called a {\it local solution} to \eqref{optimprob1} if there exists $r>0$ such that
$$
f(\bar x)\le f(x)\; \text{ \rm for all } x\in \mathcal C\cap\B(\bar x,r).
$$
In the case of $\mathcal C=\dom f,$ we say that $\bar x$ is a {\it local minimizer} of  $f$.

The following theorem provides necessary conditions for a local solution to problem \eqref{optimprob1} due to the proximal and limiting subdifferentials with respect to $\mathcal C$ of $f$.

\begin{theorem}\label{thm3}
Let $f:\mathbb R^n\to \bar \R$, and let $\mathcal C$ be closed convex. Let $\bar x\in \mathcal C$ and $f\in \mathcal F(\bar x).$ If $\bar x$ is a local solution of \eqref{optimprob1}, then
$$
0\in \partial^p_{\mathcal C}f(\bar x)\subset \partial_{\mathcal C}f(\bar x).
$$
Consequently, if $\mathcal C=\dom f$, then $0\in \bar\partial^pf(\bar x)\subset \bar\partial f(\bar x)$ whenever $\bar x$ is local minimizer of $f.$
\end{theorem}

\begin{proof}
Let $\bar x\in \mathcal C$ be a local solution to \eqref{optimprob1}. Then there exists $\theta>0$ such that
$$
f(x)\ge f(\bar x)\; \text{ \rm for all } x\in \mathcal C\cap \B(\bar x,\theta),
$$
which implies that
$$
\langle 0, x-\bar x\rangle -(f(x)-f(\bar x))\le 0\le (\Vert x-\bar x\Vert^2+\vert f(x)-f(\bar x)\vert^2)\; \forall x\in \mathcal C\cap \B(\bar x,\theta).
$$
Since $\bar x+1.0=\bar x\in \mathcal C,$ one has that $0\in \partial^p_{\mathcal C}f(\bar x)$, which is due to Proposition~\ref{pro5}~(ii) with $p=\delta=1$ and $x^*=0$.
\end{proof}

\begin{remark}
It implies from Proposition~\ref{pro4}~(ii) that $\bar\partial^p f(\bar x)\subset \partial^p f(\bar x)$. Thus $\bar\partial f(\bar x)\subset \partial f(\bar x)$, so  inclusion $0\in \bar\partial f(\bar x)$ is sharper than $0\in \partial f(\bar x).$ The two following examples present how the limiting subdifferential with respect to a set more effective than the limiting subdiferential and the projectional subdifferential at identifying the local minimizer of the problem \eqref{optimprob1}.
\end{remark}

\begin{example}
Consider the function $f:\mathbb R\to \mathbb R$ is defined by
$$
f(x)=\begin{cases}
x^2& \text{ \rm if } x\ge 0,\\
x& \text{ \rm otherwise.}
\end{cases}
$$
It is easy to see that $0\in \partial f(0)=\{0,1\}.$ Now, we divide $\dom f=\R$ by two parts $\mathcal C_1=[0,\infty)$ and $\mathcal C_2=(-\infty,0].$ It is clear that $\bar x:=0$ is local minimizer of $f$ if and only if it is local solution to both two problems: $\min\limits_{x\in \mathcal C_1} f(x)$  and $\min\limits_{x\in\mathcal C_2} f(x).$

Note that $\bar x=0$ is a local solution to problem $\min_{x\in \mathcal C_1}f(x)$ and it holds that $0\in \partial_{\mathcal C_1}f(0)=\{0\}.$ However, $0\notin \partial_{\mathcal C_2}f(0)=\{1\}.$ Thus $\bar x=0$ is not a local solution to $\min_{x\in \mathcal C_2}f(x)$, so $\bar x=0$ is not a local minimizer of $f$.
In fact, $\bar x=0$ is not really a local minimizer of $f.$
\end{example}

\begin{example}
    Consider the function $f:\mathbb R\to \bar\R$  defined by
$$
f(x)=\begin{cases}
-x& \text{ \rm if } x\ge 0,\\
\infty& \text{ \rm otherwise.}
\end{cases}
$$
Let $\mathcal C=\dom f=[0,\infty).$ By the directly computation, we have $0\in \partial\vert_{\mathcal C}f(0)=\{0\}.$ However, $0\notin \bar\partial f(0)=\{-1\}$, so $\bar x=0$ is not a local minimizer of $f$ due to Theorem~\ref{thm3}.
\end{example}

\section*{Conclusion} In this paper, we presented the versions with respect to a set of the proximal normal cone, limiting normal cone, limiting coderivative, proximal and some subdifferentials. Based on the limiting coderivative with respect to a set, we  characterized the Aubin property with respect to a set of multifunctions, which is called  the version with respect to a set of the Mordukhovich criterion.
We then applied the subdifferentials with respect to a set to provide necessary optimality conditions for local solutions to optimization problems with geometric constraints. This could  be the basis for achieving optimality conditions for optimization problems with other types of constraints as well as for developing algorithms for solving optimization problems.
\section*{Acknowledgements}
The authors thank Professor Boris Mordukhovich for his valuable comments on the first draft of this paper which significantly  improve the presentation of this paper.

\section*{Compliance with Ethical Standards}

{\bf Data Availability Statement}  Not applicable.

\noindent
{\bf Conflict of Interests/Competing Interests}  The authors declare that they have no conflict of interest.

\noindent
{\bf Funding} This article was supported by the National Natural Science Foundation of China
under Grant No.11401152.  

%%
%%===========================================================================================%%



\end{document}